\documentclass{amsart}
\usepackage{amssymb}
\usepackage{amsbsy}
\usepackage{amscd}
\usepackage{amsmath}
\usepackage{amsthm}
\usepackage{amsxtra}
\usepackage{latexsym}
\usepackage[mathscr]{eucal}
\usepackage{upgreek}
\usepackage{wasysym}
\usepackage[all,cmtip]{xy}
\usepackage[colorlinks]{hyperref}

\newcommand{\finhdual}{\overset{\circ\ \ }{\mf{h}^*}}
\newcommand{\B}{\mathcal{B}}

\newcommand{\sQche}{\overset{\circ \ \ }{Q\che}}
\newcommand{\sPche}{\overset{\circ \  }{P\che}}
\newcommand{\sPchep}{\overset{\circ \  }{P\che_+}}
\newcommand{\finPche}{\overset{\circ \ \ }{P\che}}
\newcommand{\finQche}{\overset{\circ\ }{Q\che}}

\renewcommand{\SS}[1]{US(#1)}
\newcommand{\sDelta}{\overset{\circ}{\Delta}}
\newcommand{\sDeltache}{\overset{\circ \ \ }{\Delta\che}}

\newcommand{\affrp}{{\Delta}^{re}_+}
\newcommand{\affrm}{\Delta^{re}_-}
\newcommand{\Wakimoto}{W}
\newcommand{\W}{\Wakimoto}

\newcommand{\Vs}[1]{\Irr{#1\Lam_0}}

\newcommand{\longra}{\longrightarrow}

\newcommand{\affrho}{\rho}

\newcommand{\bw}[1]{\bigwedge\nolimits^{#1}}

\newcommand{\wh}{\widehat}

\newcommand{\mc}{\mathcal}
\newcommand{\mf}{\mathfrak}

\newcommand{\on}{\operatorname}

\newcommand{\Vg}[1]{V^{#1}(\fing)}

\newcommand{\afft}{\mathfrak{t}}

\newcommand{\finb}{\overset{\circ}{\mathfrak{b}}}

\newcommand{\finn}{\overset{\circ}{\mathfrak{n}}}

\newcommand{\isomap}{{\;\stackrel{_\sim}{\to}\;}}
\newcommand{\finW}{\overset{\circ}{\mc{W}}}

\newcommand{\nc}{\newcommand}
\nc{\Hp}[1]{H^{#1}}

\newcommand{\sPi}{\overset{\circ}{\Pi}}

\newcommand{\affn}{\mathfrak{n}}
\newcommand{\sW}{\overset{\circ}{\mc{W}}}
\newcommand{\eW}{\mc{W}^e}

\newcommand{\h}{\mathfrak{h}}
\renewcommand{\a}{\mathfrak{a}}
\newcommand{\affh}{\mathfrak{h}}

\newcommand{\affg}{\mathfrak{g}}
\newcommand{\affp}{\mathfrak{p}}

\newcommand{\g}{\mathfrak{g}}
\newcommand{\fing}{\overset{\circ}{\mf{g}}}

\newcommand{\finh}{\overset{\circ}{\mf{h}}}
\newcommand{\finm}{\overset{\circ}{\mathfrak{m}}}

\newcommand{\finQ}{\overset{\circ}{Q}}

\newcommand{\Cl}{\mathscr{C}l}

\newcommand{\sP}{\overset{\circ}{P}}

\newcommand{\sI}{\overset{\circ}{I}}
\newcommand{\sroots}{\overset{\circ}{\roots}}
\newcommand{\srho}{\overset{\circ}{\rho}}
\newcommand{\srhoche}{\overset{\circ\ \ }{\rho\che}}

\newcommand{\Lamsemi}[1]{\bigwedge\nolimits^{\frac{\infty}{2}+#1}}

\newcommand{\Irr}[1]{L(#1)}

\newcommand{\BGG}{{\mathcal O}}

\newcommand{\ep}{\epsilon}

\newcommand{\N}{\mathbb{N}}
\newcommand{\Q}{\mathbb{Q}}
\renewcommand{\1}{{\mathbf{1}}}

\newcommand{\dual}[1]{{#1}^*}
\newcommand{\bra}{{\langle}}
\newcommand{\ket}{{\rangle}}

\newcommand{\roots}{\Delta}

\newcommand{\Lam}{\Lambda}

\newcommand{\lam}{\lambda}
\newcommand{\ra}{\rightarrow}
\newcommand{\+}{\mathop{\oplus}}
\newcommand{\Z}{\mathbb{Z}}

\newcommand{\inv}{^{-1}}

\renewcommand{\*}{{\otimes}}
\newcommand{\C}{\mathbb{C}}
\newcommand{\che}{^{\vee}}

\newcommand{\finp}{\overset{\circ}{\mathfrak{p}}}
\newcommand{\finl}{\overset{\circ}{\mathfrak{l}}}
\newcommand{\affl}{\mathfrak{l}}
\theoremstyle{plain}
\newtheorem{Th}{Theorem}[section]

\newtheorem{Pro}[Th]{Proposition}

\newtheorem{Lem}[Th]{Lemma}

\newtheorem{Co}[Th]{Corollary}

\theoremstyle{definition}

\theoremstyle{remark}

\newtheorem{Rem}[Th]{Remark}
\newtheorem{Conj}[Th]{Conjecture}

\newcommand{\affW}{\mc{W}}
\newcommand{\bigW}{\mathcal{W}}
\newcommand{\smallW}{\overset{\circ}{\mathcal{W}}}
\newcommand{\semiinf}{\frac{\infty}{2}}

\DeclareMathOperator{\tr}{tr}
\DeclareMathOperator{\Aut}{Aut}

\DeclareMathOperator{\ch}{ch}

\DeclareMathOperator{\End}{End}
\DeclareMathOperator{\gr}{gr}
\DeclareMathOperator{\Hom}{Hom}

\DeclareMathOperator{\ad}{ad}

\DeclareMathOperator{\haru}{span}

\title
{
Two-sided
BGG resolutions
of  admissible representations
}
\author{Tomoyuki Arakawa}
\address{Research Institute for Mathematical Sciences, Kyoto University,
 Kyoto 606-8502 JAPAN}

\email{arakawa@kurims.kyoto-u.ac.jp}

\thanks{This work is partially  supported 
by JSPS KAKENHI Grant Number
No.\ 20340007 and No.\ 23654006.}

\begin{document}
\maketitle

\begin{abstract}
We  prove the conjecture
of 
Frenkel, Kac and Wakimoto \cite{FKW92}
on the
existence of two-sided  BGG resolutions
of $G$-integrable   admissible representations
of affine Kac-Moody algebras
at fractional levels.
As an application we  establish the semi-infinite analogue of
the generalized Borel-Weil theorem \cite{Kos61}
for minimal parabolic subalgebras
which enables an inductive study of admissible representations.
\end{abstract}

\section{Introduction}
Wakimoto modules
are representations of non-twisted affine Kac-Moody algebras
introduced by Wakimoto \cite{Wak86}
in the case of $\wh{\mf{sl}}_2$  and by Feigin and Frenkel \cite{FeuFre88}
in the general case.
Wakimoto modules have useful applications in representation theory
and conformal field theory.
In these applications
it is important to have a resolution of
an irreducible highest  weight representation  $L(\lam)$
of an affine Kac-Moody algebra $\affg$
in terms of  Wakimoto modules,
that is,
a complex \begin{align*}
C^{\bullet}(\lam): \ra C^{i-1}(\lam)
\overset{d_{i-1}}{\ra} C^i(\lam)\overset{d_i}{\ra} C^{i+1}(\lam)\ra \dots
\end{align*}
with a
differential
$d_i$ which is a $\affg$-module homomorphism
such that 
$C^i(\lam)$ is a direct sum of 
Wakimoto modules and
\begin{align*}
 H^i(C^{\bullet}(\lam))=\begin{cases}
			 L(\lam)&\text{if }i=0,\\
0&\text{otherwise.}
			\end{cases}
\end{align*}
The existence of such a resolution
has been proved by Feigin and Frenkel
\cite{FeuFre90}
for any integrable representations
over arbitrary $\affg$ 
and 
by Bernard and Felder \cite{BerFel90}
and Feigin and Frenkel \cite{FeuFre90}
for any admissible representation \cite{KacWak89}
over $\wh{\mf{sl}}_2$.
In their study of $W$-algebras
Frenkel, Kac and Wakimoto \cite[Conjecture 3.5.1]{FKW92}
 conjectured  the existence of such a resolution for
any principle admissible representations 
over arbitrary $\affg$.
In this paper we prove 
the existence of a two-sided resolution  in terms of Wakimoto modules for
any
$\fing$-integrable admissible representations
over arbitrary $\affg$ (Theorem \ref{Th:two-sided BGG}),
where $\fing$ is the classical part 
of $\affg$.
For a general principal admissible representation
of $\affg$
we obtain the two-sided resolution in terms of twisted Wakimoto modules
(Theorem \ref{Th:general-twisted-BGG}).

Let us sketch the proof of our result  briefly.
 By Fiebig's equivalence \cite{Fie06}
the block of the category $\BGG$ 
of $\affg$ containing an admissible representation
$L(\lam)$
is equivalent to  the block 
containing an integrable representation\footnote{In the case 
$L(\lam)$ is a non-principal $G$-integrable admissible representation
this is a block of another Kac-Moody algebra.}. 
Therefore
an admissible representation admits a usual BGG type resolution
in terms of Verma modules
by the result of \cite{GarLep76,RocWal82}.
Hence  the idea of Arkhipov \cite{Ark96} is applicable 
in our situation:
One can  obtain  a twisted BGG resolution of 
$L(\lam)$ in terms of twisted Verma modules
by applying the twisting functor $T_w$ \cite{Ark96}
to
the BGG resolution of $L(\lam)$
as we have the ``Borel-Weil-Bott'' vanishing property \cite{AndStr03}
\begin{align*}
\mc{L}_i T_w L(\lam)\cong \begin{cases}
		    L(\lam)&\text{if }i=\ell(w),\\
0&\text{otherwise}
		   \end{cases}
\end{align*}
for $w\in \affW(\lam)$,
where $\affW(\lam)$ is the integral Weyl group of $\lam$ 
and $\ell:\affW(\lam)\ra\Z_{\geq 0}$ is the length function,
see Theorem \ref{Th:twisting-integral-Weyl-group}.
It remains to show that
one can construct an inductive system
of twisted BGG resolutions $\{B_w^{\bullet}(\lam)\}$ of $L(\lam)$ 
such that
the complex
$\lim\limits_{\longrightarrow
\atop w}B_w^{\bullet}(\lam)$ gives the
required two-sided resolution of $L(\lam)$,
see \S \ref{section:semi-infnite BGG resolution}
for the details.

We note that
by applying the (generalized)
quantum 
Drinfeld-Sokolov reduction functor \cite{FKW92,KacRoaWak03}
to the (duals of the) two-sided BGG resolutions
of admissible representations
 we obtain 
resolutions of  some of simple modules
over $W$-algebras in terms of free field realizations
due to the vanishing of the associated BRST cohomology  \cite{Ara04,Ara05,Ara07,Ara08-a,Ara09b}.
In particular we obtain  two-sided 
resolutions of all the minimal series representations
\cite{FKW92,A2012Dec}
of the  $W$-algebras associated with principal nilpotent elements
in terms of free bosonic  realizations.

As an application of
the existence of two-sided BGG resolution for admissible representations
we prove 
a semi-infinite analogue of
the generalized Borel-Weil theorem \cite{Kos61}
for minimal parabolic subalgebras 
(Theorem \ref{Th:generalized-bore-weil}).
This result  
is important since it enable an inductive study of
admissible representations, see our subsequent
paper \cite{A12-2}.

\smallskip

This paper is organized as follows.
In \S \ref{section:Semi-regular}
we  collect and prove some basic results about 
semi-infinite cohomology \cite{Feu84}
and semi-regular bimodules
\cite{Vor93}
which are needed for later use.
In particular we establish an important property
of semi-regular bimodules in Proposition \ref{Pro:iso-as-bimodules}.
In \S \ref{section:Semi-regular}
we collect basic results on the semi-infinite Bruhat  ordering 
(or the generic Bruhat ordering)
of an affine Weyl group defined by Lusztig \cite{Lus80}
and study  the semi-infinite analogue of
parabolic subgroups.
Semi-infinite Bruhat ordering is important for us since
it (conjecturally) describes the space of homomorphisms between Wakimoto
modules,
see Proposition \ref{Pro:hom-spaces}  and Conjecture \ref{conj}.
The  semi-infinite analogue of
the minimal (or maximal) length representatives
(Theorem \ref{Th:semi-infinte-minimal-length-representative})
is important 
for  describing   the semi-infinite restriction functors
 studied in \S \ref{section:Semi-infnite <restriction}.
In \S \ref{section:Wakimoto modules and twisted Verma modules}
we define Wakimoto modules and 
twisted Verma modules following \cite{Vor99}
 and study some of their basic properties.
In particular we prove the uniqueness of Wakimoto modules
which was stated in \cite{FeuFre90}
without a proof (Theorem \ref{Th:uniqueness-of-Wakimoto}).
In \S \ref{section;twisting-functors}
 we generalize the Borel-Weil-Bott vanishing
property of the twisting functor established    in \cite{AndStr03}
to the affine Kac-Moody algebra cases.
In \S \ref{section:semi-infnite BGG resolution}
we state and prove the main results of this paper.
In \S \ref{section:Semi-infnite <restriction}
we study the semi-infinite restriction functor
and
establish the semi-infinite analogue of
the generalized Borel-Weil theorem \cite{Kos61}
for minimal parabolic subalgebras.
This is a non-trivial fact since
admissible representations are not
unitarizable unless they are integrable.

\subsection*{Acknowledgments}
Some part of this work was done
while the author was visiting
 Weizmann Institute, Israel, in May 2011,
Emmy Noether Center in Erlangen,
Germany
 in June 2011, 
Isaac Newton Institute for Mathematical Sciences,
UK, in 2011,
The University of Manchester,
University of Birmingham,
The University of Edinburgh,
 Lancaster University,
York University, UK, in November 2011,
Academia Sinica, Taiwan, in December 2011.
He is grateful to those institutes
for their hospitality.

\section{Semi-regular bimodules and semi-infinite cohomology}
\label{section:Semi-regular}
\subsection{Some notation}
For $\Z$-graded vector spaces $M=\bigoplus_{n\in \Z}M_n, N
=\bigoplus_{n\in \Z}N_n$ 
with finite-dimensional homogeneous
components
let
\begin{align*}
 \mc{H}om_{\C}(M,N&)=\bigoplus_{n\in \Z}
 \mc{H}om_{\C}(M,N)_n,\\
& \mc{H}om_{\C}(M,N)_n=\{f\in \Hom_\C (M,N);f(M_i)\subset N_{i+n}\},\\
\mc{E}nd_{\C}(M)=&\mc{H}om_{\C}(M,M).
\end{align*}
We denote by
$M^*=\bigoplus_{n\in\Z}(M^*)_n$ the space
$\mc{H}om_{\C}(M,\C)$,
where $\C$ is considered as a graded vector space
concentrated in the degree $0$ component.
If $M$, $N$ are module over an algebra $A$
we denote by $\mc{H}om_A(M,N)$ the space of all $A$-homomorphisms 
in $\mc{H}om_{\C}(M,N)$.
\subsection{Semi-infinite structure}
Let $\g$ be a 
complex Lie algebra.
A {\em semi-infinite structure} \cite{Vor93}
of $\g$ is
is the following data:
\begin{enumerate}
 \item  a 
$\Z$-grading 
$\g
 =\bigoplus_{n\in \Z}\g_n$
of $\g$
with finite-dimensional homogeneous
components, 
$\dim_{\C}\g_n<\infty$ for all $n$,
\item a {\em semi-infinite $1$-cochain}
$\gamma:\g\ra \C$.
\end{enumerate} 
Here by 
a semi-infinite $1$-cochain we mean the following:
Decompose $\g$ into the direct sum of two subalgebras
\begin{align}
 &\g=\g_{+}\+ \g_{-},\label{eq:decomposition}\\
&\g_{+}=\bigoplus_{i\geq 0}\g_i,\quad
\g_{-}=\bigoplus_{i<0}\g_i.
\end{align}
A linear map $\gamma:\affg\ra \C$
is called
a semi-infinite $1$-cochain
if $\gamma$
satisfies
\begin{align*}
\gamma([x,y])=\tr ((\ad x)_{+-} (\ad y)_{-+}-(\ad y)_{+-}(\ad
 x)_{-+})\quad
\text{for }x,y\in \g,
\end{align*}
where
$(\ad x)_{\pm \mp}$ 
denotes the 
composition 
$\g_{\mp }\overset{\ad x}{\ra }\g \overset{\text{projection}}{\longra}\g_\pm$.

In the rest of this section  we assume that
$\g$
is
equipped with a semi-infinite structure
such that
 $\gamma(\sum_{i\ne 0}\g)=0$.

We denote by
$U$, $U_-$, $U_+$,
the enveloping algebras of
$\g$,
$\g_+$,
$\g_-$ by 
respectively.
These algebras inherit a  $\Z$-grading from
the corresponding Lie algebras.

Let $\tilde{\BGG}^\g$ be the category of $\Z$-graded $\affg$-modules 
$M=\bigoplus_{n\in \Z}M_n$ with
$\dim M_n<\infty$ for all $m$
on which $\bigoplus_{j>0}\g_+$ acts locally nilpotently 
and $\g_0$ acts locally finitely.

\subsection{Semi-infinite cohomology}
Choose a basis  $\{x_i; i\in \Z\}$ of 
$\g$ such that
$\{x_i; i\geq 0\}$
and $\{x_i; i<0\}$ are  bases of $\g_{+}$
and $\g_{-}$,
respectively,
and let $\{c_{ij}^k\}$ be the
structure constant:
$[x_i,x_j]=\sum_k c_{ij}^kx_k$.

Denote by  $\Cl(\g)$  the Clifford algebra 
associated with $\g\+ \g^*$,
which has the following generators and relations:
\begin{align*}
 &\text{generators: } \psi_i, \psi_i^*\quad \text{for }i\in \Z,\\
&\text{relations: }\{\psi_i,\psi_j^*\}=\delta_{i,j},\
\{\psi_i,\psi_j\}=\{\psi_i^*,\psi_j^*\}=0.
\end{align*}
Here $\{X,Y\}=XY+YX$.
The {\em space of the semi-infinite forms}
$\Lamsemi{\bullet}(\g)$ of $\g$
is by definition
the irreducible representation
of $\Cl(\g)$ generated by the vector $\1$
satisfying
\begin{align*}
 \psi_i\1=0\quad \text{for }i\geq 0
,\quad
\psi_i^*\1=0\quad \text{for }i>0.
\end{align*}
It is graded by $\deg \1=0$,
$\deg \psi_i^*=1$
and $\deg \psi_i=-1$:
 $\Lamsemi{\bullet}(\g)=\bigoplus\limits_{p\in \Z}\Lamsemi{p}(\g)$.

For 
$A\in \mc{E}nd_{\C}(\Lamsemi{\bullet}(\g))$
of degree $n$
set
\begin{align}
 :\psi_kA:=\begin{cases}
		\psi_k A   &\text{if }k<0,\\
(-1)^{n}A \psi_k&\text{if }k\geq 0,
		  \end{cases}\quad
 :\psi_k^* A:=\begin{cases}
		\psi_k^* A   &\text{if }k\leq 0,\\
(-1)^{n}A \psi_k^*&\text{if }k> 0.
		  \end{cases}
\label{eq:normal-order}
\end{align}
The following defines a $\g$-module structure  on
 $\Lamsemi{\bullet}(\g)$:
\begin{align}
 x_i\mapsto :\ad(x_i): +\gamma(x_i),
\end{align}
where
\begin{align*}
 :\ad x_i:=\sum_{j,k}c_{ij}^k :\psi_k\psi_j^*:.
\quad
\end{align*}

For $M\in \tilde{\BGG}^\g$,
define $d\in \End(M\* \Lamsemi{\bullet}(\g))$ by
\begin{align*}
& d=\sum_{i\in \Z}x_i\* \psi_i^* -1\* \frac{1}{2} \sum_{i,j,k\in \Z}c_{ij}^k
:\psi_i^*(:\psi^*_j\psi_k:):
+1\* \sum_{i\in \Z}\gamma(x_i)\psi_i^*
\end{align*}
Then
\begin{align*}
d^2=0,\quad
d(M\* \Lamsemi{p}(\g))\subset M\* \Lamsemi{p+1}(\g).
\end{align*}The  cohomology
of the complex $(M\*\Lamsemi{\bullet}(\g),d)$
is called
the {\em semi-infinite $\g$-cohomology}
with coefficients in  $M$
and denoted by
$H^{\semiinf+\bullet}(\g,M)$ (\cite{Feu84,Vor93}).

\subsection{Semi-regular bimodules}
We consider
the full dual space $\Hom_{\C}(U,\C)$ 
of $U$
as a $U$-bimodule by
$(Xf)(u)=f(uX)$, $(fX)(u)=f(Xu)$
for $X\in \g$,
$f\in \Hom_{\C}(M,\C)$, $u\in U$.
The graded duals  
$U_{\pm}^*$ of $U_{\pm}$
are $\g_{\pm}$-submodule of 
$\Hom_{\C}(U,\C)$.
By abuse of notation we denote
by $U^*$ the image 
of 
 the embedding
$U_+^*\*_{\C}U_-^*\hookrightarrow \Hom_{\C}(U,\C)$,
$f_+\* f_-\mapsto (u_-u_+\mapsto f_+(u_+)f_-(u_-))$,
$f_{\pm}\in U_{\pm}^*$, $u_{\pm}\in U$.
Then $U^*$  is a $U$-bisubmodule of $\Hom_{\C}(U,\C)$
and coincides with the image of the embedding
$U_-^*\*_{\C}U_+^*\hookrightarrow \Hom_{\C}(U,\C)$,
$f_-\* f_+\mapsto (u_+u_-\mapsto f_+(u_+)f_-(u_-))$.

Following \cite{Vor99}
define
\begin{align*}
 \SS{\g}=H^{\semiinf+0}(\g, U^*\*_{\C}U),
\end{align*}
where
$\g$ is given the opposite semi-infinite structure
and 
the semi-infinite $\g$-cohomology
is taken with respect to the diagonal left action
on 
$U^*\*_{\C}U$.
Here 
by the opposite semi-infinite structure
we mean the one
obtained 
by replacing $\g_{\pm}$
with $\g_{\mp}$
and $\gamma$ with  $-\gamma$.
The space
 $\SS{\g}$ inherits the $U$-bimodule structure from
$U^*\*U$ defined by
\begin{align*}
 X(f\* u)=-(fX)\* u,\quad (f\* u)X=f\* (uX)
\end{align*}
for
$X\in\g$,
$\in U^*$,
$u\in U$.
The $U$-bimodule $\SS{\g}$ is called the {\em semi-regular bimodule} of
$\g$.
One has
 \begin{align}
\SS{\g}\cong U_+^*\*_{U_+} U
\label{eq:ss-g-1}
\end{align}
as left
$\affg_+$-modules and  right $\affg$-modules,
and
the left $\affg$-module structure of 
$\SS{\g}$ is defined  through
the isomorphism
\begin{align}
U_+\*_{U_-}U\cong
\mc{H}om_{\C}(U_+,U)\cong   \mc{H}om_{U_-}(U,U_-\*_{\C} \C_{-\gamma})
\label{eq:ss-g-2}
\end{align}
(\cite{Vor93,Soe98,Vor99}).

 Let $M$ be a $\g$-module
and consider the following four left
$\g$-module structures  on  $\SS{\g}\*_{\C} M$:
\begin{align}
&\pi_1(X)(s\*m)=-(sX)\*m+ s\* Xm,,
\quad \pi_2(X)(s\* m)=(Xs)\* m,
\label{eq:pi1andpi2}\\ & 
\pi_1'(X)(s\*m)=-(sX)\* m,\quad
\pi_2'(X)(s\* m)=(Xs)\*m+s\*(Xm),
\end{align}
for $X\in \g$,
$s\in \SS{\g}$,
$m\in M$.
Clearly, the two actions $\pi_1$ and $\pi_2$
(resp.\ $\pi_1'$ and $\pi_2'$) commute.

\begin{Pro}[cf.\ {\cite[6.4]{ArkGai02}}]
\label{Pro:iso-as-bimodules}
For $M\in \tilde{\BGG}^\g$
the
two $U$-bimodule structures 
$(\pi_1,\pi_2)$ and $(\pi_1', \pi_2') $ on $\SS{\g}\*_\C M$
are equivalent.
Namely there exists a linear isomorphism
$\Phi:\SS{\g}\*_\C M\isomap\SS{\g}\*_\C M$
such that
$\Phi\circ \pi_i'(X)=\pi_i(X)\circ \Phi$
for $i=1,2$, $X\in \g$.
\end{Pro}
\begin{proof}
Define the linear isomorphism
\begin{align*}
 \tilde\Phi_1:U^*\*_{\C}U\*_{\C}M
\isomap U^*\*_{\C}U\*_{\C}M
\end{align*}
 by $\Phi_1(f\* u\* m)=f\*(\Delta(u)(1\*m))$
for $f\in U^*$, $u\in U$,
$m\in M$,
where $\Delta:U\ra U\*_{\C} U$ is the coproduct.
We have
\begin{align*}
& \tilde\Phi_1\circ  \pi_{2,L}(X)=(\pi_{2,L}(X)+\pi_{3,L}(X))\circ\tilde \Phi_1\\
& \tilde\Phi_1\circ  (\pi_{2,R}(X)+\pi_{3,R}(X))=\pi_{2,R}(X)\circ
 \tilde \Phi_1,
\end{align*}
where
$\pi_{i,L}$ (resp.\ $\pi_{i,R}$) denotes the left action (resp.\ the
 right action) of $\g$  on the $i$-th factor of
$U^*\* U\* M$,
and
$M$ is considered as a right $\g$-module
by the action $mx=-xm$
for $m\in M$, $x\in \g$.

Next consider
the graded dual
$M^* =\bigoplus_{n} (M^*)_n$ as a right module
by the action $(fX)(m)=f(Xm)$.
Let
\begin{align*}
 \Psi:U^*\*_{\C}M
\isomap U^*\*_{\C}M
\end{align*}
be the linear isomorphism
defined 
by
$\Psi(f\* m)(u\*g)=(f\* m)((1\* g)\Delta(u))$
for 
$f\in U^*$, $m\in M$,
$u\in U$, $g\in M^*$,
where
$M$ is identified with $(M^*)^*$.
Extend this to 
the linear isomorphism
\begin{align*}
\tilde \Phi_2:U^*\*_{\C}U\*_{\C}M
\isomap U^*\*_{\C}U\*_{\C}M
\end{align*}
 by setting $\tilde \Phi_2(f\* u\* m)=\sum_i f_i \* u\* m_i$ if $\Psi(f\* m)=\sum_i
 f_i\* m_i$
with $f_i\in U^*$, $m_i\in M$.
Then
\begin{align*}
&\tilde \Phi_2\circ \pi_{1,R}(X)=(\pi_{1,R}(X)+\pi_{3,R}(X))\circ \tilde
 \Phi_2,\\
&\tilde \Phi_2\circ (\pi_{1,L}(X)+\pi_{3,L}(X))=\pi_{1,L}(X)\circ \tilde
 \Phi_2.
\end{align*}
Set
\begin{align*}
 \tilde \Phi=\tilde \Phi_2\circ\tilde  \Phi_1
:U^*\*_{\C}U\*_{\C}M
\isomap U^*\*_{\C}U\*_{\C}M.
\end{align*}
Then
\begin{align}
& \tilde \Phi\circ (\pi_{1,L}(X)+\pi_{2,L}(X))=\tilde \Phi_2\circ (\pi_{1,L}(X)+\pi_{2,L}(X)+\pi_{3,L}(X))\circ
 \tilde \Phi_1 \label{eq:fix1}
\\
&=(\pi_{1,L}(X)+\pi_{2,L}(X))\circ \tilde \Phi,\nonumber
\\&\tilde \Phi\circ (\pi_{2,R}(X)+\pi_{3,R}(X))=\tilde \Phi_2\circ
 \pi_{2,R}(X)\circ \tilde \Phi_1=
\pi_{2,R}(X)\circ\tilde \Phi, \label{eq:fix2}\\&
\tilde \Phi\circ \pi_{1,R}(X)=\tilde \Phi_2\circ \pi_{1,R}(X)\circ
 \tilde \Phi_1=(\pi_{1,R}(X)+\pi_{3,R}(X))\circ \tilde \Phi.
\label{eq:fix3}
\end{align}
By  \eqref{eq:fix1} and the definition of $\SS{\g}$,
$\tilde \Phi$ gives rise to a linear isomorphism
 \begin{align*}
  \Phi: \SS{\g}\*_{\C} M\isomap \SS{\g}\*_{\C}M.
 \end{align*}
Moreover
$\Phi$ satisfies the required  properties
by \eqref{eq:fix2} and $\eqref{eq:fix3}$.
\end{proof}

\subsection{Semi-infinite induction}
Let $\h=\bigoplus_{n\in \Z}\h_n$ be a graded Lie subalgebra of $\g$
such that
$\gamma|_{\affh}$  is a semi-infinite 
$1$-cochain of $\affh$.
Following \cite{Vor99}
we define the {\em semi-induced $\g$-module}
$\on{S-ind}^{\g}_{\h}M$
as
\begin{align*}
\on{S-ind}^{\g}_{\h}M
:=H^{\semiinf+0}(\h,\SS{\g}\*_{\C} M),
\end{align*}
where
$\SS{\g}\*_\C M$ is considered as an $\h$-module by
 the action $\pi_1$ defined in (\ref{eq:pi1andpi2}).
The space
$\on{S-ind}^\g_{\h}M$ inherits the structure of a $\g$-module from
the action
$\pi_2$  defined in (\ref{eq:pi1andpi2}).

\begin{Lem}
The assignment 
$\on{S-ind}^{\g}_\h 
: M\mapsto \on{S-ind}^{\g}_\h 
M$ 
defines an exact functor from $\tilde{\BGG}^{\h}$ to $\tilde{\BGG}^{\g}$.
\end{Lem}
\begin{proof}
Clearly 
$\on{S-ind}M$ is an object of $\tilde{\BGG}^{\g}$
since $\SS{\g}\*_{\C}M$ is. 
By Proposition \ref{Pro:iso-as-bimodules}
we may replace the actions of $\pi_1$
and $\pi_2$ on $\SS{\g}\*_{\C} M$
with $\pi_1'$ and $\pi_2'$,
simultaneously. 
It follows that
\begin{align}
H^{\semiinf+\bullet}(\h,\SS{\g}\*_{\C}M)
\cong H^{\semiinf+\bullet}(\h,\SS{\g})\*_{\C}M.
\label{eq:s-induction}
\end{align}
Since $\SS{\g}$ is   free over $\h_-$ and cofree over $\h_+$,
$H^{\semiinf+i}(\h,\SS{\g})=0$ for $i\ne 0$
by \cite[Theorem 2.1]{Vor93}.
(Note that 
the spectral sequence
on \cite{Vor93} converges since
the complex $\SS{\g}\* \Lamsemi{\bullet}(\h)
$
is a direct sum of finite-dimensional
subcomplexes consisting of homogeneous vectors.)
We have shown that
the functor $\on{S-ind}^{\affg}_{\affh}$ is exact.
\end{proof}

In the case that $\h=\g$ and $\gamma_0=\gamma$,
we have the following assertion.
\begin{Pro}[{\cite[(1.9)]{Vor99}}]\label{Pro:identity-functor}
The functor $\on{S-ind}^{\g}_\g:\tilde{{\BGG}}^\g\ra \tilde{\BGG}^\g$ is 
isomorphic to the identify functor.
\end{Pro}
\begin{proof}
As 
$H^{\semiinf+0}(\g,\SS{\g})$ is isomorphic to the trivial
 representation
$\C$ of $\g$
(\cite[Theorem 2.1]{Vor93}),
\eqref{eq:s-induction} gives the $\g$-module isomorphism
$ \on{S-ind}^{\g}_{\g}M\cong M$ as required.
\end{proof}

\subsection{}
Suppose that  $\g$ admits a decomposition
\begin{align*}
 \g=\a\+ \bar \a
\end{align*}
with graded subalgebras $\a$ and $\bar \a$
such that the restrictions
 $\gamma|_{\a}$
and $\gamma|_{\bar \a}$ of $\gamma$
are semi-infinite
$1$-cochains of $\a$ and $\bar \a$,
respectively.
\begin{Lem}\label{Lem:PBW}
 $\SS{\g}\cong \SS{\a}\*_{\C}\SS{\bar \a}$
as left $\a$-modules and right $\bar \a$-modules.
\end{Lem}
\begin{proof}
We have $U_+^*\cong U(\a_+)^*\*_{\C}U(\bar \a_+)^*$
as left $\a_+$-modules and right $\bar \a_+$-modules.
Consider the composition
\begin{align*}
\SS{\a}\*_{\C}\SS{\bar \a}\isomap (U(\a_-)\*_\C U(\a_+)^*)\*_\C
(U(\bar \a_+)^*\*_{\C}U(\bar \a_-))
\\ \isomap  U(\a_+)\*_{\C}U_+^*\*_{\C}U(\bar
 \a_+)\ra \SS{\g},
\end{align*} 
where the last map is the multiplication map.
From  the description
(\ref{eq:ss-g-1}), (\ref{eq:ss-g-2}) 
of the $\g$-bimodule structure of semi-regular bimodules
one sees that
the image of the above map
is stable under the left and the right action of $\g$ on $\SS{\g}$.
Hence the image must coincides with $\SS{\g}$
since it contains $U_+^*$.
By the equality of the graded dimensions
it follows that
above map is 
 an isomorphism.
\end{proof}

\begin{Lem}\label{Lem:Shapiro}
For $M\in \tilde{\BGG}^{\bar \a}$,
$\on{S-ind}^{\g}_{\bar \a}M\cong \SS{\a}\*_{\C}M$ as
$\a$-modules,
where $\a$ acts only on the first factor $\SS{\a}$ of 
$\SS{\a}\*_{\C}M$.
\end{Lem}
\begin{proof}
We have
\begin{align*}
\on{S-ind}^{\g}_{\bar \a}(M)
\cong 
H^{\semiinf+0}(\bar \a,\SS{\a}\*_{\C}\SS{\bar \a}\*_{\C}M)\\
\cong \SS{\a}\*_{\C}\on{S-ind}_{\bar\a}^{\bar \a}(M)
\cong
\SS{\a}\*_{\C} M
\end{align*}by 
Lemmas \ref{Pro:identity-functor}
and  \ref{Lem:PBW}.
\end{proof}

\section{Semi-infinite Bruhat ordering}
\label{section:Semi-infintie Bruhar ordering} 
\subsection{Affine Kac-Moody algebras
and affine Weyl groups}
\label{subsection:KM-and-affine-Weyl}
We first fix some notation which are used for the rest of the paper.

Let $\fing$ be a finite-dimensional  complex simple Lie algebra,
and
fix a Cartan subalgebra
$\finh$ of $\fing$.
Let $\sDelta\subset \finhdual$
be the set of roots
of $\fing$.
Choose 
a subset 
$\Delta_+\subset \finhdual$  of positive roots
and the set 
 $\sPi=\{\alpha_i; i\in \sI\}\subset \Delta_+$, 
$\sI=\{1,2,\dots l\}$,
of simple roots.
Let $\theta$ be the highest root,
$\theta_s$ the highest short root,
$\Delta_-=-\Delta_+$,
$\srho=\frac{1}{2}\sum\limits_{\alpha\in \sDelta_+}\alpha$.
Let $\finQ=\sum\limits_{\alpha\in \sDelta}\Z \alpha\subset \finhdual$,
the root lattice of $\fing$,
Set
$\finn=\bigoplus\limits_{\alpha\in \sroots_+}\fing_{\alpha}$,
$\finn_-=\bigoplus\limits_{\alpha\in \sroots_-}\fing_{\alpha}$,
where
$\fing_{\alpha}$ is the root space of $\fing$
with root $\alpha$.
We have the triangular decomposition
\begin{align*}
 \fing=\finn_-\+ \finh\+ \finn.
\end{align*}

Let 
$(~|~)$  be the normalized invariant bilinear form of $\fing$.
We identify $\finh$ with $\finhdual$ using $(~|~)$.
Let $\sDeltache=\{\alpha\che: \alpha\in \sDelta\}$,
the set of coroots ,
$\finQche=\sum\limits_{\alpha\in \sDelta}\Z\alpha\che\subset \finh
=\finhdual$,
the coroot lattice of $\fing$,
$\srhoche=\frac{1}{2}\sum\limits_{\alpha\in \sDelta_+}\alpha\che$,
where 
$\alpha\che=2\alpha/(\alpha|\alpha)$.

Let $\finW\subset GL(\finhdual)$ be the Weyl group of $\fing$,
$s_{\alpha}\in \finW$ be the reflection corresponding to $\alpha\in
\Delta$:
$s_{\alpha}(\lam)=\lam-\lam(\alpha\che)\alpha$.

Let
$\affg$ be
the affine Kac-Moody algebra
associated with  $\fing$:
\begin{align*}
 \affg=\fing[t,t\inv]\+ \C K.
\end{align*}
The commutation  relations of $\affg$
are given by
\begin{align*}
 [xt^m,yt^n]=[x,y]t^{m+n}+m\delta_{m+n,0}(x|y)K,
\quad 
[K,\affg]=0.
\end{align*}
We consider $\fing$ as a subalgebra of $\affg$
by the natural embedding
$\fing\hookrightarrow \affg$, $x\mapsto x t^0$.
Let
$\affh=\finh\+ \C K$,
the 
Cartan subalgebra of $\affg$,
$\affh^*=\finhdual\+ \C \Lam_0$ the dual of $\affg$.
Here $\finhdual(K)=\Lam_0(\finh)=0$,
$\Lam_0(K)=1$.
The number 
$\bra \lam,K\ket $ is called the {\em level} of $\lam$.

Let
$\Delta^{\vee,re}_+
=\sDeltache\sqcup \{\alpha\che+ nK;\alpha\in
\sDelta_{long}
, n\in \N
\}\sqcup \{\alpha\che +r\che nK;\alpha\in \sDelta_{short}, n\in \Z\}$
be the set of positive real coroots,
$\Delta^{\vee,re}_-=-\Delta^{\vee,re}_+$ the set of
negative real coroots,
$\Delta^{\vee,re}=\Delta^{\vee,re}_+\sqcup \Delta^{\vee,re}_-$,
the set of real coroots,
where 
$\sDelta_{long}$ 
and $\sDelta_{short}$ 
are the sets of long roots and short roots of $\fing$,
respectively.
Let 
$\Pi\che=\{\alpha_1\che,\dots,
\alpha_{\ell}\che,\alpha_0\che=-\theta+K\}$,
the set of simple coroots.

We have the following action
$\alpha\mapsto t_{\alpha}$ of $\finhdual$
on $\affh^*$:
\begin{align*}
 t_{\alpha}(\lam)=\lam+\bra \lam,K\ket \alpha,
\quad \lam\in \dual{\affh}.
\end{align*}
Dually
$\finh$ acts on $\affh$
as $t_{\alpha}(h)=h-\bra \alpha,h\ket K$.
For a subset $L\subset \finhdual$
let
$t_L=\{t_{\alpha}|\alpha\in L\}$.
The Weyl group of $\affg$,
or the {\em affine Weyl group} $\affW$ of $\finW$,
is the semidirect product
\begin{align*}
 \affW=\finW\ltimes t_{\sQche}.
\end{align*}
Below we often write $\alpha$ for $t_{\alpha}\in t_{\sQche}$.

The {\em extended affine  Weyl group} $\eW$ of 
the semidirect product
\begin{align*}
\eW=\finW\ltimes t_{\overset{\circ\ }{P\che}}
\end{align*}
where
$\overset{\circ\ }{P\che}=\{\lam\in \finh;
\bra\alpha,\lam\ket \in \Z\
\text{for all }\alpha\in \sDelta\}$,
  the coweight lattice of $\fing$.
We have
\begin{align*}
 \eW=\eW_+
\ltimes \affW,
\end{align*}
where
$\eW_+$ subgroup of $\eW$ consisting of elements which fix the set
$\Pi^{\che}$.

To define the set of roots of $\affg$,
one needs to
enlarge $\affh$ by a one-dimensional space
$\C D$ by letting
$\tilde{\affh}=\affh\+ \C D=\finh\+ \C K\+ \C D$
and
extend
the bilinear form
$(~|~)$ from $\finh$ to $\tilde{\affg}$ by letting
$(K|\finh)=(D|\finh)=(K|K)=(D|D)=0$ and $(D|K)=1$.
We identify
$\dual{\affh}$ with the subspace of $\dual{\tilde{\affh}}$
consisting of elements which vanishes on $D$.
Thus,
\begin{align*}
 \dual{\tilde{\affh}}=\dual{\affh}\+ \C \delta=
\finhdual\+ \C \Lam_0 \+ \C \delta,
\end{align*}
where
$\delta$ is defined by
\begin{align*}
 \delta|_{\affh}=0,\quad \bra \delta,D\ket=1.
\end{align*}
The action of $\affW$
on $\dual{\affh}$ is extended to $\dual{\tilde{\affh}}$
by
\begin{align*}
& w(\delta)=\delta\quad w\in \finW,\\
&
t_{\alpha}(\lam)=\lam+\bra \Lam,K\ket\alpha
-(\bra \lam,\alpha\ket+\frac{(\alpha|\alpha)}{2}
\bra \lam,K\ket)\delta,\quad 
\lam \in \dual{\affh}.
\end{align*}
For $\lam\in \dual{\tilde{\affh}}$
let
$ \bar \lam\in \finhdual$
be its restriction to $\finh$.

Let $\affrp=\sDelta_+\sqcup
\{\alpha+n\delta; \alpha\in\sDelta,\ n\in \N\}$,
the set of positive real roots,
$\affrm=-\affrp$,
$\Delta^{re}=\affrp\sqcup \affrm$ the set of real roots,
$\Pi=\{\alpha_0=-\theta+\delta,\alpha_1,\dots, \alpha_{\ell}\}$
the set of simple roots.
The reflection
$s_{\alpha}$ corresponding
$\alpha=\bar{\alpha}+n\delta\in \Delta^{re}$
is given by
$s_{\alpha}=t_{-n\bar{\alpha}\che}s_{\bar{\alpha}}$.
We set
$s_i=s_{\alpha_i}$  for $i\in I:=\{0,1,\dots, l\}$.

\subsection{Twisted Bruhat ordering}
Let $\ell:\eW\ra \Z_{\geq 0}$
 the length function:
$\ell(w)=\sharp (\affrp
\cap
w(\affrm))$.
We have
\begin{align}
  &\ell(t_{\mu}y)=\sum_{\alpha\in \Delta_+\cap y(\Delta_+)}|(\alpha|\mu)|
+\sum_{\alpha\in \Delta_+\cap y(\Delta_-)}|1-(\alpha|\mu)|
\label{eq:formula-for-usual-length}
\end{align}
for $\mu\in \finPche$,
$y\in \finW$.

The {\em twisted length function} \cite{Ark96}
$\ell^y:\eW\ra \Z$
with the twist 
$y\in \eW$
is defined 
 by
\begin{align*}
 \ell^y(w)=
\sharp (\affrp\cap
w(\affrm)\cap y(\affrp))
-
\sharp (\affrp\cap
w(\affrm)\cap y(\affrm)).
\end{align*}
\begin{Lem}\label{Lem:twisted-length}
Let $w,y\in \eW$.
\begin{enumerate}
 \item
Suppose that $\ell(ys_i)=\ell(y)+1$ for $i\in I$.
Then
 \begin{align*}
 \ell^{ys_i}(w)=
\begin{cases}
 \ell^{y_{}}(w)&
\text{if }w\inv y(\alpha_i)
\in \affrp
,\\
 \ell^{y_{}}(w)-2&
\text{if }w\inv y(\alpha_i)\in \affrm.
\end{cases}
\end{align*}
\item $\ell^y(w)=\ell(y\inv w)-\ell(y\inv)$. 
\end{enumerate}
\end{Lem}
\begin{proof}
(i)
The assertion follows from the definition and the fact that
\begin{align*}
\affrp\cap ys_i(\affrm)
=\affrp\cap y(\affrm)\sqcup
\{y(\alpha_i)\}\quad \text{if }\ell(ys_i)=\ell(y)+1.
\end{align*}(ii)
We prove by induction on $\ell(y)$.
If $\ell(y)=0$ then
$\ell^y(w)=\ell(w)=\ell(y\inv w)$.
Suppose that
$\ell(ys_i)=\ell(y)+1$.
If $w\inv y(\alpha_i)\in \affrp$
then $\ell(s_iy\inv w)=\ell(y\inv w)+1$.
Hence by (i) and induction hypothesis,
\begin{align*}
 \ell^{ys_i}(w)=\ell^y(w)=\ell(y\inv w)-\ell(y\inv)
=\ell(s_iy\inv w)-\ell(s_i y\inv).
\end{align*}
 If $w\inv y(\alpha_i)\in \affrm$
then $\ell(s_iy\inv w)=\ell(y\inv w)-1$.
Again by (i) and induction hypothesis,
\begin{align*}
 \ell^{ys_i}(w)=\ell^y(w)-2=\ell(y\inv w)-2-\ell(y\inv)
=\ell(s_iy\inv w)-\ell(s_i y\inv).
\end{align*}
This completes the proof.
\end{proof}
For $w_1,w_2,y\in \affW$
and $\gamma\in \Delta^{re}$,
write  $w_1\overset{\gamma}{\underset{y}{\longrightarrow}
}w_2$
if $w_1=s_{\gamma}w_2$ and 
$\ell^{y}(w_1)>\ell^{y}(w_2)$.
Below, we shall often omit the symbol $\gamma$ above the arrow.
Also,
we shall omit the symbol $y$ under the arrow
if $y=1$.
By Lemma \ref{Lem:twisted-length} (ii)
we have
\begin{align}
w_1\overset{y(\gamma)}{\underset{y}{\longrightarrow}
}w_2\quad \iff \quad
y\inv w_1\overset{\gamma}{{\longrightarrow}
}y\inv w_2.
\label{eq:def-of-twisted-bruhat-order}
\end{align}
In particular
for $\beta=y(\alpha_i)\in \affrp$,
$\alpha_i\in \Pi$,
and $w_1,w_2\in \affW$ such that $\ell^y(w_2)-\ell^y(w_1)=1$
we have the equivalence
\begin{align}
 \xymatrix{
& s_{\beta}w_1 & &s_{\beta}w_1 
\ar[rd]_{y}&\\
w_1 
\ar[ru]_{y}
\ar[rd]_{y} & & \iff  & 
& 
 s_{\beta}w_2\\
& w_2 &  & w_2 \ar[ru]_{y}} 
\label{eq:BGG-equivalence-y}
\end{align}
 by \cite[Lemma 11.3]{BerGelGel75}.

Define
$w\succeq_y w'$ if 
there exists a sequence $w_1,w_2,\dots,w_k$ of elements of $\affW$
such that
\begin{align*}
 w\underset{y}{\longrightarrow}w_1\underset{y}{\longrightarrow}
w_2\underset{y}{\longrightarrow}\dots 
\underset{y}{\longrightarrow}w_k
\underset{y}{\longrightarrow}w'.
\end{align*}
Note that
\begin{align}
 w\succeq_y
 w' &\iff 
 y\inv w\succeq y\inv w',
\label{eq:twisted-Bruhar-orderint}
\end{align} 
by \eqref{eq:def-of-twisted-bruhat-order},
where $\succeq=\succeq_1$,  the  usual Bruhat ordering of $\affW$.
It follows that
$\succeq_y$
 defines a partial ordering of $\affW$.

We will use the symbol 
$w\rhd_y w'$
 to denote a covering in the twisted Bruhat order
$\succeq_{y}$.
Thus $w\rhd_y w'$ means that 
$w \succeq_y w'$ and 
$\ell^y(w)=\ell^y(w')+1$.

\subsection{Semi-infinite Bruhat ordering}
Define the {\em semi-infinite length}
 \cite{FeuFre90}
$\ell^{\semiinf}(w)$ of $w\in \eW$
by 
\begin{align*}
 \ell^{\semiinf}(w)
&=\sharp \{\alpha\in \affrp\cap w( \affrm);
\bar \alpha\in \sDelta_+\}
-\sharp \{\alpha\in \affrp\cap 
w( \affrm);
\bar \alpha\in \sDelta_-\}.
\end{align*}
We have
\begin{align}
 \ell^{\semiinf}(t_{\lam}y)=\ell(y)-2(\srho|\lam)
\label{eq:formula-for-semiinfinite-length-function}
\end{align} 
for $\lam\in \finPche$,
$w\in \sW$.

Set
\begin{align*}
 \sPchep=\{\lam\in \finPche;
\alpha(\lam)\geq 0  \text{ for all }\alpha\in \sDelta_+\},
\end{align*}
We say
that
$\lam\in \sPchep$ is sufficiently large
if $\alpha(\lam)$ if sufficiently large
for all $\alpha\in \sDelta_+$.

By (\ref{eq:formula-for-usual-length})
and (\ref{eq:formula-for-semiinfinite-length-function})
we have the following assertion.
\begin{Lem}\label{Lem:semi-infnite-length}
 $\ell^{\semiinf}(w)=\ell^{\lam}(w)=-\ell^{-\lam}(w)$
for a sufficiently large $\lam\in \sPchep$ .
\end{Lem}

We write 
\begin{align*}
  w_1\overset{\gamma}{\underset{\semiinf}{\longrightarrow}
}w_2
\end{align*}
for $w_1,w_2\in \affW$
and $\gamma\in \Delta^{re}$
if $w_1=w_2s_{\gamma}$ and 
$\ell^{\semiinf}(w_1)<\ell^{\semiinf}(w_2)$.
(We shall often omit the symbol $\gamma$ above the arrow.)
Define
$w\succeq_{\semiinf} w'$ if 
there exists a sequence $w_1,w_2,\dots,w_k$ of elements of $\bigW$
such that
\begin{align*}
 w\underset{\semiinf}{\longrightarrow}w_1\underset{\semiinf}{\longrightarrow}
w_2\underset{\semiinf}{\longrightarrow}\dots 
\underset{\semiinf}{\longrightarrow}w_k
\underset{\semiinf}{\longrightarrow}w'.
\end{align*}
By Lemma \ref{Lem:semi-infnite-length}
\begin{align*}
 w\succeq_{\semiinf}
 w' &\iff 
w'\succeq_{t_{\lam}} w
\quad\text{for a sufficiently large } \lam\in \sPchep,\\
 &\iff 
w\succeq_{t_{-\lam}} w'
\quad\text{for a sufficiently large } \lam\in \sPchep.
\end{align*} 
It follows that
$\succeq_{\semiinf}$ defines a partial ordering of $\affW$.
Following Arkhipov \cite{Ark96}, 
we call it the {\em semi-infinite Bruhat ordering} on $\affW$.
By \cite[Claim 4.14]{Soe97}
the semi-infinite Bruhat ordering coincides with the
{\em generic Bruhat ordering} defined by Lusztig \cite{Lus80}.

We will use the symbol 
$w\rhd_{\semiinf} w'$
 to denote a covering in the twisted Bruhat order
$\succeq_{\semiinf}$.
Thus $w\rhd_{\semiinf} w'$ means that 
$w \succeq_{\semiinf} w'$ and 
$\ell^{\semiinf}(w)=\ell^{\semiinf}(w')-1$.

\subsection{Semi-infinite analogue of parabolic subgroups
and minimal (maximal) length representatives}
\label{subsection:Semi-infinite analogue of parabolic subgroups}
Let $S$ be a subset of $\sPi$,
$\sDelta_S$
the
subroot system of $\sDelta$
generated by $\alpha_i\in S$,
$\sDelta_S=\bigsqcup_{i=1}^r \sDelta_{S,i}$
the decomposition into
the simple subroot systems $\sDelta_{1,S},\dots, \sDelta_{r,S}$.
Let
$\theta_i$
the 
 longest root of  $\sDelta_{S,i}$.

Set
\begin{align*}
&  \Delta_S=\{\alpha+n\delta\in \Delta^{re};
\alpha\in \sDelta_S, n\in \Z\},
\quad \affW_S=\bra s_{\alpha};
\alpha\in \Delta_S\ket\subset \affW.
\end{align*}
Then $\Delta_S$
is a subroot system
of $\Delta^{re}$
isomorphic to the affine
root system associated with $\sDelta_S$.
Put $\Delta_{S,+}=\Delta_S\cap \affrp$,
the set
of positive root of  $\Delta_S$.
Then  $\Pi_S=S\sqcup \{-\theta_1+\delta,\dots,-\theta_s+\delta\}
$ is
a set of simple roots  of $\Delta_S$.
We have
$\affW_S= \sW_S\ltimes t_{\sQche_S}
$,
where 
$ \overset{\circ\ }{Q\che_S}=\sum\limits_{\alpha\in \sDelta_S}\Z\alpha\che$.
By (\ref{eq:formula-for-semiinfinite-length-function}),
 the restriction of the semi-infinite length function
to $\bigW_S$ coincides with the 
semi-infinite length function of the affine Weyl group $\affW_S$.

Define
\begin{align*}
 \affW^S=\{w\in \affW;w\inv(\Delta_{S,+})\subset \affrp\}.
\end{align*}

\begin{Th}[\cite{Pet}]\label{Th:semi-infinte-minimal-length-representative}  
The multiplication map
$\bigW_S\times \bigW^S\ra \bigW$,
$(u,v)\mapsto uv$, is a bijection.
Moreover, we have
\begin{align*}
 \ell^{\semiinf}(uv)=\ell^{\semiinf}(u)+\ell^{\semiinf}(v)
\quad \text{for }u\in \bigW_S, 
\ v\in \bigW^S.
\end{align*}
\end{Th}
\begin{proof}
First,
we show the injectivity of  the multiplication map.
Suppose that
$u_1v_1=u_2 v_2$ with $u_i\in \bigW_S$,
$v_i\in \bigW^S$.
Then 
$v_1=uv_2$
with
$u=u_1\inv u_2\in \bigW_S$.
If $u\ne 1$ then there exists
$\alpha \in {\Delta}_{S,+}$
such that
$u\inv(\alpha)\in -{\Delta}_{S,+}$.
But then $v_2\in \bigW^S$ implies that
$v_1\inv(\alpha)=v_2\inv u\inv(\alpha)\in \affrm$,
and this contradicts
that $v_1\in \bigW^S$.
Hence $u_1=u_2$,
and so $v_1=v_2$.

Second, we
show that
the multiplication map
$\bigW_S\times \bigW^S\ra \bigW$ is surjective.
We will prove by  induction
on $\sharp (w \inv({\Delta}_{S,+})\cap \affrm)$
that 
there exists $u\in 
\bigW_S$ such that $u\inv w\in \bigW^S$.
If $\sharp (w \inv({\Delta}_{S,+})\cap \affrm)=0$, 
$w\in \bigW^S$  there is nothing to show.
Suppose that
$\sharp (w \inv({\Delta}_{S,+})\cap \affrm)>0$.
Then
there exists $\beta\in {\Pi}_S$ such that
$w\inv(\beta)\in \affrm$.
Indeed,
any element $\alpha\in{\Delta}_{S,+}$
is expressed as
$\alpha=\sum_{\beta \in {\Pi}_S}n_{\beta}\beta$
with $n_{\beta}\in\Z_{\geq 0}$.
Thus
$w\inv(\alpha)=\sum_{\beta \in {\Pi}_S}n_{\beta}w\inv(\beta)\in \affrm$
implies that
one of $w\inv(\beta)$ must belong to $\affrm$.
Now because 
$(s_{\beta}w) \inv({\Delta}_{S,+})
=w\inv s_{\beta}({\Delta}_{S,+})
=w\inv ({\Delta}_{S,+}\backslash \{\beta\}\sqcup \{-\beta\})
=w\inv ({\Delta}_{S,+})\backslash \{w\inv(\beta)\}
\sqcup \{-w\inv(\beta)\}$,
\begin{align*}
(s_{\beta}w) \inv({\Delta}_{S,+})
\cap \affrm
=w \inv({\Delta}_{S,+})\cap \affrm
\backslash\{w\inv(\beta)\}.
\end{align*}
Hence
by applying the induction hypothesis to
$s_{\beta}w$
we find an element $u\in \bigW_S$ such that
$u\inv s_{\beta}w\in \bigW^S$.

Finally, we prove the equality of the semi-infinite length.
By (\ref{eq:formula-for-semiinfinite-length-function}),
we have
$\ell^{\semiinf}(t_{\mu}w)=\ell^{\semiinf}(t_{\mu})
+\ell^{\semiinf}(w)$ for any $\mu\in \sQche$.
Hence we may assume that
$u\in \smallW_S$.
We will prove  by induction on the length
$\ell(u)$ of $u\in \smallW_S$
that 
$\ell^{\semiinf}(uv)=\ell^{\semiinf}(u)+\ell^{\semiinf}(v)$
for any
$v\in \bigW^S$.
Suppose that $\ell(u)=1$,
so that $u=s_i$ for some $\alpha_i\in S$.
Let
$v=t_{\mu}y\in \affW^S$ with $\mu\in \sQche$,
$y\in \sW$.
Note that
$v\in \bigW^S$ is equivalent to that
\begin{align}
\text{if $\alpha\in \sDelta_{S,+}$ then }\quad
 \alpha(\mu)=\begin{cases}
	      0&\text{if $y\inv(\alpha)\in \sDelta_+$,}\\
	      1&\text{if $y\inv(\alpha)\in \sDelta_-$}.
	     \end{cases}
\label{eq:formala-of-elements-of-W1}
\end{align}
Since
\begin{align*}
 \ell^{\semiinf}(s_i t_{\mu}y)=
\ell(t_{s_i(\mu)}s_i y)=
\ell(s_i y)-2(\rho|\mu-\alpha_i(\mu)\alpha_i\che)
=\ell(s_i y)-2(\rho|\mu)+2 \alpha_i(\mu),
\end{align*}
  (\ref{eq:formala-of-elements-of-W1})
implies that
$\ell^{\semiinf}(s_iv)= 
\ell^{\semiinf}(v)+1$.
Next let $u=s_i u_1\in \sW_S$ with 
$u_1\in \sW_S$,
$\alpha_i\in S$,
$\ell(u)=\ell(u_1)+1$,
so that $u_1\inv(\alpha_i)\in \sDelta_+$.
Let  $v=t_{\mu}y\in \affW^S$ as above.
We have
\begin{align*}
 \ell^{\semiinf}(uv)
=\ell^{\semiinf}(t_{s_i u_1(\mu)}s_i u_1y)=\ell(s_i u_1y)
-2(\rho|s_iu_1(\mu)).
\end{align*}
If
$\ell(s_i u_1y)=\ell(u_1y)+1$,
then $\sDelta_+\ni (u_1y)\inv(\alpha_i)=y\inv(u_1\inv(\alpha_i))$.
Hence 
$(\mu|u_1\inv(\alpha_i))=0$ by (\ref{eq:formala-of-elements-of-W1}),
which means
$s_iu_1(\mu)=u_1(\mu)$.
By the induction hypothesis,
this proves that 
$ \ell^{\semiinf}(uv
)=\ell^{\semiinf}(u)+\ell^{\semiinf}(v)$.
If
$\ell(s_i u_1y)=\ell(u_1y)-1$,
then $\sDelta_-\ni (u_1y)\inv(\alpha_i)=y\inv(u_1\inv(\alpha_i))$.
So (\ref{eq:formala-of-elements-of-W1})
gives
$(\mu|u_1\inv(\alpha_i))=1$,
which means
$s_iu_1(\mu)=u_1(\mu)-\alpha_i\che$.
By the induction hypothesis,
this proves that 
$ \ell^{\semiinf}(uv
)=\ell^{\semiinf}(u)+\ell^{\semiinf}(v)$
as required.

\end{proof}

\section{Wakimoto modules and twisted Verma modules}
\label{section:Wakimoto modules and twisted Verma modules}
  \subsection{Twisting functors and twisted Verma modules}
By abuse of  notation we denote also by $w$ a Tits lifting of  $w\in \eW$
to $\Aut (\affg)$.

For each $w\in \affW$
the twisting functor
$T_w: \BGG^{\g}\ra \BGG^{\g}$
is defined 
as follows (\cite{Ark96}):
Let
$\affn_w=\affn_-\cap w\inv (\affn_+)$
and set $N_w=U(\affn_w)$.
Put
\begin{align*}
S_w=U\*_{N_w}N_w^*.
\end{align*}
The space
$S_w$ has a $U$-bimodule structure,
which is described as follows:
Let $f\in \affn_-\backslash \{0\}$,
and set $U_{(f)}=U\*_{\C[f]}\C[f,f\inv]$.
Then $U_{(f)}$ is an associative algebra which contains $U$ as a
subalgebra.
We set $S_f=U_{(f)}/U$.
Choose a filtration 
$\affn_w=F^0\supset F^1\supset \dots \supset F^r\supset 0$,
$r=\ell(w)$,
consisting of ideals $F^p\subset \affn_w$ of codimension $p$.
If $f_p\in F^{p-1}\backslash F^p$ we have an isomorphism of
$U$-bimodules
\begin{align}
 S_w=S_{f_1}\*_U S_{f_2}\*_{U}\dots \*_{U}S_{f_r}.
\label{eq:iso-of-Sw}
\end{align}
We have
\begin{align}
S_w\cong N_w^*\*_{N_w}U
\label{eq:Sw-another}
\end{align}
as right $U$-modules
and left $N_w$-modules.
Put
\begin{align*}
\1_w^*=f_1\inv\* f_2\inv \* \dots \*f_{r}\inv\in S_w.
\end{align*}

For
$M\in \BGG^{\g}$
define
\begin{align*}
 T_w(M)=\phi_w(S_w\*_{U(\affg)}M),
\end{align*}
where 
$\phi_w$ means that the action of $\affg$
is twisted by the  automorphism $w$ of $\affg$.
This define a right exact functor  $T_w: \BGG^{\g}\ra \BGG^{\g}$
such that
\begin{align}
& T_{ws_i}\cong T_w T_{i}\quad
\text{if }\alpha_i\in \Pi \text{ and }\ell(ws_i)=\ell(w)+1,
\label{eqt:wisting-functor-construction}
\end{align}
where $T_i=T_{s_i}$.

The functor $T_w$ admits a right adjoint  
functor 
$G_w$ in the category $\BGG^{\affg}$ (\cite[\S 4]{AndStr03}):
\begin{align*}
 G_w (M)=
\mc{H}om_{U}(S_w,\phi_w\inv (M)).
\end{align*}

It is straightforward 
to extend
the definition of
 $T_w$ and $G_w$ to
 $w\in
\affW^e$ (\cite{Ara04}).

The following assertion follows in the same manner
as \cite[Theorem 2.1]{Soe98}.
 \begin{Lem}
\label{Lem:twistinf-functor-TG=1 cases}
Let $M\in \BGG^{\affg}$,
$w\in \eW$
\begin{enumerate}
 \item Suppose that
$M$
is free over $\affn_w$.
Then $M\cong G_wT_w(M)$.
 \item Suppose that
$M$ is cofree over $w(\affn_w)$.
Then $M\cong T_wG_w(M)$.

\end{enumerate}
 \end{Lem}
For $\lam\in \dual{\affh}$,
let 
$M(\lam)$ be the Verma module of $\affg$
with highest weight $\lam$.
Set
\begin{align*}
M^w(\lam)=  T_w M(w\inv\circ \lam).
\end{align*}
The $\affg$-module
$M^{w}(\lam)\in \BGG^{\affg}$ is called
the {\em twisted Verma module}  $M^w(\lam)$ with highest weight $\lam$ and
twist $w\in \affW^e$.
Note that by \eqref{eq:Sw-another}
we have 
\begin{align}
 M^w(\lam)_{\mu}\cong \phi_w(N_w^*\*_{N_w}U(\affn_-))_{\mu-\lam}
\cong 
\left(U(w(\affn_-)\cap \affn_+)^*\*_{\C}
U(w(\affn_-)\cap \affn_-)\right)_{\mu-\lam}
\label{eq:twisted-Verma-as-vec}
\end{align} as $\affh$-modules.
Hence\begin{align*}
\ch M^w(\lam)=\ch M(\lam).
\end{align*}
In particular
$W(\lam)$ is an object of $\BGG^\g$.

By  Lemma \ref{Lem:twistinf-functor-TG=1 cases} (1)
we have
\begin{align*}
M(\mu)\cong G_w M^w(w\circ \mu).
\end{align*}
Hence
the functor $T_w$ gives the isomorphism
\begin{align}
\Hom_{\g} (M(\lam), M( \mu))\isomap 
\Hom_{\g} (M^w(w\circ \lam), M^w(w\circ \mu))
\label{eq:hom-spaces-of-twisted-Verma}
\end{align}
for 
$\lam,\mu\in \dual{\affh}$.

We have \cite[Proposition 6.3]{AndLau03}
\begin{align}
 M^w(\lam)\cong M(\lam)
\quad\text{if }\bra \lam+\rho,\alpha\che\ket\not\in \N
\quad\text{for all }\alpha\in \affrp\cap w(\affrm).
\label{eq:twisted-Verma=Verma}
\end{align}

\subsection{Hom spaces between twisted Verma modules}
For $\lam\in \dual{\affh}$
let $\Delta(\lam)$
and $\affW(\lam)$ be its 
{\em integral root system} and 
{\em integral Weyl group},
respectively:
\begin{align*}
& \Delta(\lam)=\{\alpha\in \Delta^{re};
\bra \lam+\rho,\alpha\che\ket\in \Z\},\\
&\affW(\lam)=\bra s_{\alpha}; \alpha\in \Delta(\lam)\ket\subset \affW.
\end{align*}
Let 
$\Delta(\lam)_+=\Delta(\lam)\cap \affrp$
the set of positive roots
of $\Delta(\lam)$,
$\Pi(\lam)\subset \Delta(\lam)_+$  the set
of simple roots of $\Delta(\lam),$
 $\ell:\affW(\lam)\ra \Z_{\geq 0}$  the length function.

For
$y\in \affW(\lam)$
the  twisted length function 
$\ell^{y}$
and
the
twisted Bruhat ordering
$\succeq_{\lam,y}$
are defined for  $\affW(\lam)$.
We will use the symbol 
$w\rhd_{\lam,y} w'$
 to denote a covering in the twisted Bruhat order
$\succeq_{\lam,y}$.

Recall that
a weight  $\lam\in \dual{\affh}$ is called  {\em regular dominant}
if
$\bra \lam+\rho,\alpha\che\ket\not\in \{0,-1,-2,\dots \}$
for all $\alpha\in \affrp$.
It is called
{\em regular anti-dominant}
if 
$\bra \lam+\rho,\alpha\che\ket\not\in \{0,1,2,\dots \}$
for all $\alpha\in \affrp$.

\begin{Th}\label{Th:hom-spaces-twiste-Verma}
Let 
$w,w',y\in \affW(\lam)$.
\begin{enumerate}
 \item 
If $\lam$ is regular dominant
then
\begin{align*}
\dim_{\C}  \Hom_{\affg}(M^y(w\circ \lam), M^{y}(w'\circ \lam))
=\begin{cases}
  1&\text{if }w\succeq_{\lam,y} w',\\
0&\text{otherwise.}
 \end{cases}
\end{align*}
 \item 
If $\lam$ is regular anti-dominant
then
\begin{align*}
\dim_{\C}  \Hom_{\affg}(M^y(w\circ \lam), M^{y}(w'\circ \lam))
=\begin{cases}
  1&\text{if }w\preceq_{\lam,y} w',\\
0&\text{otherwise.}
 \end{cases}
\end{align*}

\end{enumerate}
\end{Th}
\begin{proof}
(i) By \eqref{eq:hom-spaces-of-twisted-Verma}
the assertion follows from 
\eqref{eq:twisted-Bruhar-orderint}
and \cite[Proposition 2.5.5 (ii)]{KasTan98}.
Proof of (ii) is similar.
\end{proof}

\subsection{Wakimoto modules}
Let $\affg$,
$\affh$ be as in \S \ref{subsection:KM-and-affine-Weyl},
and let us consider the $\Z$-grading 
of $\affg$ with $\affg_0=\affh$,
$\affg_1=\bigoplus_{\alpha\in \Pi}\affg_{\alpha}$,
where $\affg_{\alpha}$ is the root space of $\affg$ of  root $\alpha$.
Let $\rho=\srho+  h\che\Lam_0\in \dual{\affh}$,
where $h\che$ is the dual Coxeter number of $\fing$.
Then
$\bra \rho,\alpha\che\ket=1$ for all $\alpha\in \Pi$
and $2\rho$ 
define a semi-infinite $1$-cochain of $\affg$ \cite{Ark97}.

For any $\affh$-module $M$
we set $M_{\mu}=\{m\in M; hm=\mu(h)m\ \text{for all }h\in \affh\}$.

Let $\BGG^{\g}$ be the full subcategory 
of $\tilde{\BGG}^{\g}$ consisting of modules on which $\affh$ acts
semisimply.
The formal character of $M\in \BGG^{\affg}$
is defined by
$\ch M=\sum_{\mu\in \dual{\affh}}(\dim_{\C}M_{\mu})e^{\mu}$.

Let $\BGG_k^{\affg}$ be the full subcategory of $\BGG^{\affg}$
consisting of objects of level $k$,
where
a $\affg$-module $M$ is said to be of level $k$
if $K$ acts as the multiplication by $k$.

Let
$L\finn$,
$L\finn_-$,
$\a$ and $\bar \a$ 
be 
 graded subalgebras of $\affg$
defined by
\begin{align*}
&L\finn=\finn[t,t\inv],
\quad L\finn_-=\finn_-[t,t\inv],\\
& \a=L\finn\+\finh[t\inv]t\inv,\quad \bar \a=L\finn_-\+\finh[t]\+ \C K.
\end{align*}
Then
$0=2\rho|_{L\finn}=2\rho|_{L\finn_-}=2\rho|_{\a}$ gives   semi-infinite 
$1$-cochains of 
$L\finn$,
$L\finn_-$, $\a$,
and $2\rho|_{\bar \a}$ gives a 
semi-infinite $1$-cochain of 
$\bar \a$.

Following \cite{Vor99}
we define the {\em Wakimoto module} $W(\lam)$ of $\affg$
with highest weight $\lam\in \dual{\affh}$ by
 \begin{align*}
  W(\lam)=\on{S-ind}^{\affg}_{\bar \a}\C_{\lam},
 \end{align*}
where 
$\C_{\lam}$ is 
the one-dimensional representation
of $\affh$ corresponding to  $\lam$ regarded as a $\bar \a$-module
by the natural projection $\bar \a\twoheadrightarrow \affh$.
By Lemma \ref{Lem:Shapiro}
we have
\begin{align}
& W(\lam)\cong \SS{\a}\text{ as $\a$-modules},
\label{eq:wakimoto-1}
\end{align}
and hence
\begin{align}
&H^{\semiinf+i}(\a,W(\lam))\cong \begin{cases}
		      \C_{\lam}& \text{if }i=0,\\
0&\text{otherwise}
		     \end{cases}\text{ as $\affh$-modules},
\label{eq:homological-property-of-Wakimoto}
\\
&\ch W(\lam)=\ch M(\lam).
 \end{align}
In particular
$W(\lam)$ is an object of $\BGG^\g$.

Theorem \ref{Th:uniqueness-of-Wakimoto} below
shows that
the above definition of Wakimoto module coincides with that of Feigin
and 
Frenkel \cite{FeuFre90,Fre05}.

\subsection{Wakimoto modules as inductive limits of
twisted Verma modules}
\label{subsection:inducitive-limits}

Let $y,w,u\in \affW$
such that $w=y u$ 
and 
$\ell(w)=\ell(y)+\ell(u)$.
Then $T_w=T_yT_u$
and 
$S_w\cong S_y\*_{U}\phi_y(S_u)$. 
Let 
\begin{align*}
j_{w,y}:S_y\longrightarrow S_w
\end{align*}
be the homomorphism
of left $U$-modules
which maps $s\in S_y$ to $s\* \1_u^*
\in S_y\*_{U}\phi_y(S_u)=S_w$.
Define
$\nu_{w,y}^\lam\in \Hom_{\affg}(M^y(\lam), M^w(
\lam))$ by
\begin{align*}
 \nu_{w,y}^{\lam}(s\* v_{y\inv \circ \lam})=j_{w,y}(s)\* v_{w\inv \circ
 \lam}\quad \text{for }s\in S_y,
\end{align*}
where $v_{\mu}$ denotes the highest weight vector of $M(\mu)$
for $\mu\in \dual{\affh}$.
Then 
\begin{align*}
 \Hom (M^{y}(\lam),M^w(\lam))=\C  \nu_{w,y}^{\lam}
\end{align*}
by   (\ref{eq:hom-spaces-of-twisted-Verma}).
 We have
\begin{align}
 \nu^{\lam}_{w_3,w_2}\circ \nu^{\lam}_{w_2,
w_1}=\nu^{\lam}_{w_3,w_1}
\end{align}
if $w_3=w_2u_2$, $w_2=w_1 u_1$
with 
$\ell(w_1)=\ell(w_2)+\ell(u_2)$,
$\ell(w_2)=\ell(w_1)+\ell(u_1)$.

Let $\{\gamma_1,\gamma_2,\dots\}$ be a sequence
in $\sPchep$ such that
$\gamma_i-\gamma_{i-1}\in \sPchep$
and 
$\lim\limits_{n\ra \infty}\alpha(\gamma_n)=\infty$ for all $\alpha \in
\sroots_+$.
Then 
$t_{-\gamma_{i+1}}
=t_{-\gamma_i}t_{-(\gamma_{i+1}-\gamma_{i})}$
with
$\ell(t_{-\gamma_{i+1}})=\ell(t_{-\gamma_{i}})+\ell(t_{-(\gamma_{i+1}-\gamma_{i})})$
for all $i$.
It follows that
$\{M^{-\gamma_n}(\lam): \nu_{-\gamma_m,-\gamma_n}^{\lam}\}$ forms an
inductive system of $\affg$-modules.

\begin{Pro}[{\cite[Lemma 6.1.7]{Ark96}}]
\label{Pro:Wakimoto-as-limit}
There is an isomorphism of $\affg$-modules
\begin{align*}
W(\lam)\cong \lim\limits_{\longrightarrow\atop n}M^{-\gamma_n}(\lam).
\end{align*}
 \end{Pro}
\begin{proof}
For the reader's convenience we shall give a proof of Proposition
\ref{Pro:Wakimoto-as-limit}
here.
Set
$W(\lam)'=\lim\limits_{\longrightarrow\atop n}M^{-\gamma_n}(\lam)$.
First note that
\begin{align*}
t_{-\gamma_i}(\affn_{-\gamma_i})
=& t_{-\gamma_i}(\affn_-)\cap \affn_+
=
\haru_{\C}\{x_{\alpha}t^n;\alpha\in \Delta_+,\
0\leq n<\alpha(\gamma_i)\},
\\
&t_{-\gamma_i}(\affn_-)\cap \affn_-
=(\finh\+\finn)[t\inv]t\inv\+
\haru_{\C}\{x_{-\alpha}t^{-n};\alpha\in \Delta_+,\
n>\alpha(\gamma_i)\},
\end{align*}
where $x_{\alpha}$ is a root vector of $\fing$
of root $\alpha$.
Thus
we have
$
t_{-\gamma_1}(\affn_{-\gamma_1})
\subset t_{-\gamma_2}(\affn_{-\gamma_2})\subset \dots\subset  \a_+$
and
$\a_+=\bigcup_{i\geq 1}t_{-\gamma_i}(\affn_{-\gamma_i})$.
The
 map
 $j_{-\gamma_i,-\gamma_j}
:S_{-\gamma_i}\ra S_{-\gamma_j}$
restricts
to the embedding 
 $j_{-\gamma_i,-\gamma_j}
:N_{-\gamma_i}^*\hookrightarrow N_{-\gamma_j}^*$
for $i<j$,
and we have 
\begin{align*}
 U(\a_+)^*\cong \lim\limits_{\longrightarrow \atop i}\phi_{-\gamma_i}(N_{-\gamma_i}^*)
\end{align*}
as left $\a_+$-modules.
Let $j_{-\gamma_i}:\phi_{-\gamma_i}(N_{-\gamma_i}^*)\hookrightarrow
 U(\a_+)^*$ be the embedding of left $\phi_{-\gamma_i}(N_{-\gamma_i})
$-modules under the above identification.

Since $t_{-\gamma_i}(\affn_{-\gamma_i})=
\haru_{\C}\{x_{\alpha}t^{-n};\alpha\in \Delta_+,\
0< n\leq \alpha(\gamma_i)\}\subset \a$,
\begin{align*}
W(\lam)\cong T_{-\gamma_i}G_{-\gamma_i}(W(\lam))
\end{align*}
by Lemma \ref{Lem:twistinf-functor-TG=1 cases} (ii).
Hence 
\begin{align*}
 \Hom_{\affg}(M^{-\gamma_i}(\lam), W(\lam)) 
\cong \Hom_{\affg}(M(t_{\gamma_i}\circ \lam), G_{-\gamma_i}(W(\lam))).
\end{align*}
As
$\ch G_{-\gamma_i}(W(\lam))=\ch M(t_{\gamma_i}\circ \lam)$,
there exists 
a unique $\affg$-module homomorphism
 $\psi_i:
M(t_{\gamma_i}\circ \lam)\ra G_{{-\gamma_i}}(M)$
which sends
$v_{t_{\gamma_i}\circ \lam}$ to $w_i$,
a vector of 
$G_{-\gamma_i}(W(\lam))$ of weight $t_{\gamma_i}\circ \lam$.
Up to a non-zero constant multiplication,
$w_i$ equals to the
 the element of 
$G_{-\gamma_i}(W(\lam))= \mc{H}om_{N_{{-\gamma_i}}}(N_{{-\gamma_i}}^*,
 \phi_{{-\gamma_i}}^{-1}(W(\lam)))$
which sends
$f\in N_{t_{-\gamma_i}}^*$ to $j_{-\gamma_i}(f)\* 1_{\lam}
\in \SS{\a}\* \C_{\lam}=W(\lam)$.
The corresponding homomorphism
$T_{{-\gamma_i}}(\psi_i)
:M^{-\gamma_i}(\lam)\ra W(\lam)$
is given by
\begin{align}
 T_{{-\gamma_i}}(\psi_i)(f\*v_{t_{\gamma_i}\circ \lam})=
j_{-\gamma_i}(f)\* 1_{\lam} \quad\text{for 
$f\in N_{-\gamma_i}^*$ 
}\label{eq:map-from-wakimoto}.
\end{align}
It follows that 
$T_{{-\gamma_i}}( \psi_j) 
\circ
 \nu_{\gamma_j,\gamma_i}^{\lam}=T_{{-\gamma_i}}(\psi_{i})$ for $i<j$ ,
and  the sequence 
$\{T_{{-\gamma_i}}( \psi_j) \}$
 yields a $\affg$-module homomorphism 
\begin{align*}
\Phi: W(\lam)'=\lim\limits_{\longrightarrow \atop i}M^{-\gamma_i}(\lam)
\longrightarrow W(\lam).
\end{align*}
Fix $\mu\in \dual{\affh}$.
Since $W(\lam)\cong US(\a)$ as an 
$\a$-module,
it follows from \eqref{eq:twisted-Verma-as-vec}
that
 $T_{-{\gamma_i}}$
restricts to the
isomorphism
$M^{-\gamma_i}(\lam)_{\mu}\isomap 
W(\lam)_{\mu}$ for a sufficiently large $i$.
This completes the proof.
\end{proof}
\subsection{Endmorphisms   of Wakimoto modules}
\begin{Pro}\label{Pro:Wakimoto-is-invariant-under-twisting}
Let $\alpha\in \sPchep$,
$\lam\in \dual{\affh}$.
 \begin{enumerate}
\item 
$T_{-\alpha} W(\lam)\cong W(t_{-\alpha}\circ  \lam)$.
  \item 
$G_{-\alpha}W(\lam)\cong W(t_{\alpha} \circ \lam)$.
 \end{enumerate}
\end{Pro}
\begin{proof}
(i)
Let $\{\gamma_1,\gamma_2,\dots\}$ be a sequence
in $\sPchep$ such that
$\gamma_i-\gamma_{i-1}\in \sPchep$
and 
$\lim\limits_{n\ra \infty}\beta(\gamma_n)=\infty$ for all $\beta \in \sroots_+$.
Set
$\gamma'_i=\gamma_i+\alpha$.
Then the sequence 
 $\{\gamma_1',\gamma_2',\dots\}$ satisfies the same property.
Hence
by Proposition \ref{Pro:Wakimoto-as-limit}
and the fact that a homology functor commutes with inductive limits
we have
$
 T_{-\alpha} W(\lam)\cong 
T_{-\alpha}(\lim\limits_{\longrightarrow }
M^{-\gamma_i}(\lam))
=
\lim\limits_{\longrightarrow }T_{-\alpha}M^{-\gamma_i}(\lam)
=\lim\limits_{\longrightarrow }T_{-\alpha}T_{-\gamma_i}M(t_{\gamma_i}\circ \lam)
= \lim\limits_{\longrightarrow }T_{-\gamma_i'}M(t_{\gamma_i}\circ \lam)
=\lim\limits_{\longrightarrow } M^{-\gamma_i'}(t_{\alpha}\circ \lam)
\cong W(t_\alpha\circ \lam)
$.
(ii)
Since
$\affn_{t_{-\alpha}}\subset \a_-$,
$W(\lam)$ is free over
$\affn_{t_{-\alpha}}$.
Hence
$W(t_\alpha\circ \lam)=G_{-\alpha}T_{-\alpha}W(t_\alpha\circ\lam)\cong 
G_{-\alpha}W(\lam)$
by Lemma \ref{Lem:twistinf-functor-TG=1 cases}
and (i).
\end{proof}
\begin{Co}\label{Co:hom-space-of-Wakimoto}
Let 
$\alpha\in \sPchep$.
The functor $G_{-\alpha}$ gives the isomorphism
\begin{align*}
\Hom_{\affg}(W(\lam), W(\mu))\cong \Hom_{\affg}(W(t_{\alpha}\circ \lam),
W(t_{\alpha}\circ \mu)).
\end{align*}
for  $\lam,\mu\in \dual{\affh}$.
\end{Co}

\begin{Pro}\label{Pro:endmorphisms-of-Wakimoto}
 For $\lam\in \dual{\affh}$
we have
$\End_{\affg}(W(\lam))=\C$.
\end{Pro}
\begin{proof}
Let $\{\gamma_1,\gamma_2,\dots, \}$ be in Subsection \ref{subsection:inducitive-limits}.
Then
 \begin{align*}
 \End_{\g}(W(\lam))
=
\Hom_{\affg}(\lim_{\longrightarrow \atop i}M^{-\gamma_i}(\lam),
W(\lam))
\quad
\text{(by Proposition \ref{Pro:Wakimoto-as-limit})}\\
=\lim_{\longleftarrow\atop i}
\Hom_{\g} (M^{-\gamma_i}( \lam), W( \lam))
\cong 
\lim_{\longleftarrow\atop i}\Hom_\affg(M(t_{\gamma_i}\circ
  \lam),G_{-\gamma_i}W( \lam))
\\ \cong
\lim_{\longleftarrow\atop i}\Hom_{\affg}(M(t_{\gamma_i}\circ
  \lam),W(t_{\gamma_i}\circ \lam))
\quad\text{(by Proposition \ref{Pro:Wakimoto-is-invariant-under-twisting})}.
\end{align*}
As we have seen in the proof of Proposition \ref{Pro:Wakimoto-as-limit},
the space $\Hom_{\affg}(M(t_{\gamma_i}\circ
  \lam),W(t_{\gamma_i}\circ \lam))$ 
is one-dimensional
and 
$\nu_{-\gamma_m,\gamma_n}^{\lam}$
induces the isomorphism 
\begin{align*}
\Hom_\affg(M^{-\gamma_m}(\lam), W(\lam))
\isomap \Hom_{\affg}(M^{-\gamma_n}(\lam),W(\lam)).
\end{align*}
This completes the proof.
\end{proof}

\subsection{Uniqueness of Wakimoto modules}
\label{subsection:{Uniqueness of Wakimoto modules}}
 A finite filtration
$0=M_0\subset M_1\subset M_2\subset M_{r}=M$ 
of a $\affg$-module $M$ is called
a{\em Wakimoto flag} if 
each successive quotient 
$M_{i}/M_{i-1}$  is isomorphic to $W(\lam_i)$ 
for some $\lam_i$.

\begin{Th}\label{Th:uniqueness-of-Wakimoto}
Suppose that
$k$ is non-critical, that is,
$k\ne -h\che$.
For  
an object 
 $M$ of $\BGG^{\g}_k$ 
the following conditions are equivalent.
\begin{enumerate}
 \item $M$ admits a Wakimoto flag.
\item $H^{\semiinf+i}(\a,M)=0$ for $i\ne 0$
and $H^{\semiinf+0}(\a,M)$ is finite-dimensional.
\end{enumerate}
If this is the case
the multiplicity
 $(M: W(\lam))$ of $W(\lam)$
in a Wakimoto flag of $M$
equals to 
$\dim H^{\semiinf+0}(\a, M)_{\lam}$.
In particular
if \begin{align*}
H^{\semiinf+i}(\a,M)\cong \begin{cases}
		      \C_{\lam}& \text{if }i=0,\\
0&\text{otherwise}
		     \end{cases}
\end{align*}
as $\affh$-modules,
$M$ is isomorphic to $\Wakimoto(\lam)$.
\end{Th}
The proof of Theorem 
\ref{Th:uniqueness-of-Wakimoto}
will be given in Subsection \ref{subsection:uniquness}.

We put on record some of  consequences of Theorem
\ref{Th:uniqueness-of-Wakimoto}:
\begin{Pro}
A tilting module  in $\BGG^{\affg}$  at a non-critical level
 admits a Wakimoto flag.
\end{Pro}
\begin{proof}
 By definition a tilting module $M$ admits both a Verma flag and a dual
 Verma flag. It follows that $M$ is free over $\affn_-$ and cofree over
 $\affn_+$.
In particular $M$ is free over $\finn[t\inv]t\inv$ and cofree over
 $\finn[t]$.
Hence by \cite[Theorem 2.1]{Vor93},
we have
 $H^{\semiinf+i}(\a,M)
=0$ for $i\ne 0$.
The fact that
  $H^{\semiinf+0}(\a,M)$ is finite-dimensional follows from
 the Euler-Poincar\'{e} principle.
\end{proof}

\begin{Pro}\label{Pro:Wakimoto-at-negative-level}
 Suppose that
$\bra \lam+\rho,K\ket\not \in \Q_{\geq 0}$.
Then
$
 W(t_{\alpha}\circ \lam)\cong M(t_{\alpha}\circ \lam)$
for a sufficiently large $\alpha\in \sPchep
$.
\end{Pro}
\begin{proof}
Let 
$\alpha$ be sufficiently large.
By the hypothesis
$\bra t_{\alpha}(\lam+\rho),\beta\che\ket \not \in \N$ for 
all $\beta\in \affrp$ such that $\bar \beta\in \sDelta_+$.
It follows from \cite[Theorem 3.1]{Ara04}
that 
$M(t_{\alpha}\circ \lam)$ is cofree over
$\finn[t]=\a_+$.
Because $M(t_{\alpha}\circ \lam )$ is obviously
free over $\a_-$
we have
$H^{\semiinf+i}(\a,M(t_\alpha\circ \lam))\cong
\begin{cases}
 \C_{t_{\alpha}\circ \lam}&\text{for }i=0,\\
0&\text{otherwise.}
\end{cases}$
\end{proof}

The following assertion follows from
 Proposition \ref{Pro:Wakimoto-at-negative-level}
and Corollary \ref{Co:hom-space-of-Wakimoto}.
\begin{Pro}\label{Pro:hom-spaces}
 Let $\lam,\mu\in \dual{\affh}$ be of level $k$,
and suppose that
$k+h\che \not \in \Q_{\geq 0}$.
Then
\begin{align*}
\Hom_{\affg}(W(\lam), W(\mu))\cong \Hom_{\affg}(M(t_{\alpha}\circ \lam),
M(t_{\alpha}\circ \mu))
\end{align*}
for a sufficiently large $\alpha\in \sPchep$.
In particular if $\lam\in \dual{\affh}$ is integral, regular
 anti-dominant, 
then
\begin{align*}
\dim_\C \Hom_{\affg}(W(w\circ \lam),W(y\circ \lam))=\begin{cases}
						  1&\text{if
						     }w\preceq_{\semiinf}y\\
0&\text{else}
						 \end{cases}
\end{align*}
for $w,y\in\affW$.
\end{Pro}

 \begin{Conj}\label{conj}
Let $\lam\in \dual{\affh}$  be integral,  regular dominant.
Then
\begin{align*}
\dim_\C \Hom_{\affg}(W(w\circ \lam),W(y\circ \lam))=\begin{cases}
						  1&\text{if
						     }w\succeq_{\semiinf}y\\
0&\text{else}
						 \end{cases}
\end{align*}
for $w,y\in\affW$.
 \end{Conj}
In Theorem \ref{Th:two-sided BGG} below
we prove Conjecture \ref{conj} in the case that
$w\rhd_{\semiinf}y$ (in a slightly more general setting).

\subsection{Proof of 
 Theorem \ref{Th:uniqueness-of-Wakimoto}}\label{subsection:uniquness}
Let
\begin{align*}
\mc{H}=\finh[t,t\inv]\+ \C K\subset \affg,
\end{align*}
the Heisenberg subalgebra.
Denote by  $\pi_{\lam}$ the  
irreducible representation of  $\mc{H}$
 with highest weight $\lam$.
We have
$\pi_{\lam}\cong U(\finh[t\inv]t\inv)$ as a module over
$\finh[t\inv]t\inv\subset \mc{H}$
provided that $\lam(K)\ne 0$.

For  $M\in\BGG^{\g}_k$
one knows that 
$H^{\semiinf+\bullet}(L\finn,M)$
is naturally an $\mc{H}$-module
of level $k+h\che$ (\cite{FeuFre90}).
\begin{Lem}\label{Lem:characterization1}
Let $M$ be an object  of $\BGG^{\affg}_k$
with $k\ne -h\che$.
Then the following conditions are 
equivalent:
\begin{enumerate}
 \item  $H^{\semiinf+i}(\a,M)=0$ for $i\ne 0$;
\item $H^{\semiinf+i}(L\finn,M)=0$ for $i\ne 0$.
\end{enumerate}
\end{Lem}
\begin{proof}
The assumption
that $k\ne -h\che$ implies that
$H^{\semiinf+\bullet}(L\finn, M)
$
is semi-simple as an $\mc{H}$-module
and is a direct sum of
$\pi_{\mu}$s.
Consider the Hochschild-Serre
 spectral sequence
for the ideal $L\finn\subset \a$
to compute $H^{\semiinf+\bullet}(\a,M)$.
By definition,
we have
\begin{align*}
 E_2^{p,q}=
\begin{cases}
 H_{-p}(\finh[t\inv]t\inv,H^{\semiinf+q}(L\finn,M))&\text{for }p\leq 0,
\\ 0&\text{for }p>0.
\end{cases}
\end{align*}
By the above mentioned fact
$H^{\semiinf+q}(L\finn, M)$  is free over $U(\finh[t\inv]t\inv)$.
Hence
\begin{align*}
 E_2^{p,q}=
\begin{cases}
 H^{\semiinf+q}(L\finn,M))
/\finh[t\inv]t\inv(H^{\semiinf+q}(L\finn,M)))
&\text{for }p= 0.
\\ 0&\text{for }p\ne 0.
\end{cases}
\end{align*}
Therefore the spectral sequence collapses 
at $E_2=E_{\infty}$,
and 
$H^{\semiinf+i}(\a, M)=0$ for $i\ne 0$ if and only if
$ H^{\semiinf+i}(L\finn,M)=0$ for $i\ne 0$.
This completes the proof.
\end{proof}
\begin{Pro}\label{Pro:characterization-2}
 Let $M$ be 
an object of $\BGG_k$
at a non-critical level $k$
 such that
 $H^{\semiinf+i}(\a,M)=0$ for $i\ne 0$.
Then
\begin{align*}
 M\cong \SS{\a}\*_{\C} H^{\semiinf+0}(\a, M)
\end{align*}
as $\a$-modules
and $\affh$-modules,
where
$\a$ acts only on the first factor $\SS{\a}$
and 
$\affh$ acts as
$h(s\* m)=\ad(h)(s)\* m+s\* hm$.
\end{Pro}
\begin{proof}
 By Proposition \ref{Pro:identity-functor}
it suffices to show that
$\on{S-ind}^{\a}_{\a}M\cong \SS{\a}\*_{\C}H^{\semiinf+0}(\a, M)$.
 As in the proof of Lemma \ref{Lem:characterization1},
we shall consider 
the Hochschild-Serre
 spectral sequence
for the ideal $L\finn\subset \a$
to compute
$H^{\semiinf+\bullet}(\a, \SS{\a}\* M)$.
By definition we have
\begin{align}
& E_1^{\bullet,q}=H^{\semiinf+q}(L\finn,\SS{\a}\*_{\C} M)\*_{\C}
 \bw{\bullet}(\finh[t\inv]t\inv),
\label{eq:E1-temp}
\\
& E_2^{p,q}=
H_{-p}(\finh[t\inv]t\inv,
H^{\semiinf+q}(L\finn,\SS{\a}\*_\C M)).
\end{align}

To compute the $E_1$-term 
set
\begin{align*}
 F^p \SS{\a}=\bigoplus_{\bra 
  \mu,\srhoche\ket \geq p}\SS{\a}_{\mu},
\end{align*}
where $\SS{\a}$ is considered as an
$\affh$-module by the adjoint action.
Then
\begin{align*}
&\SS{\a}= F^0 \SS{\a}\supset F^1\SS{\a}\supset \dots,\quad
\bigcap F^p \SS{\a}=0,\\
& 
F^p \SS{\a}\cdot L\finn  \subset F^{p+1}\SS{\a}.
\end{align*}
Define the 
filtration
$F^{\bullet}(\SS{\a}\*_{\C} M \*_{\C} \Lamsemi{\bullet}(L\finn))$
 by setting
\begin{align*}
F^p (\SS{\a}\*_{\C}M\*_{\C} \Lamsemi{\bullet}(L\finn))
=
F^p \SS{\a}\*_{\C} M\*_{\C} \Lamsemi{\bullet}(L\finn).
\end{align*}
This defines 
 a
decreasing, weight-wise regular filtration of the 
complex.
Consider the associated  spectral sequence
$
E_r'\Rightarrow H^{\semiinf+\bullet}(L\finn, \SS{\a}\*_{\C} M)$.
Because the associated graded space $\gr \SS{\a}$ with respect
to this filtration is a trivial
$L\finn$-module
the  $E_1$-term of the spectral sequence $E_r'$
is 
isomorphic to
$\SS{\a}\*_{\C} H^{\semiinf+\bullet}(L\finn, 
  M)$.
Hence
by the hypothesis  and Lemma \ref{Lem:characterization1}
the spectral sequence $E'_r$
collapses at $E_1'=E_{\infty}'$
and we obtain 
the isomorphism
of $\affh$-modules 
\begin{align}
 H^{\semiinf+i}(L\finn, \SS{\a}\*_{\C}
  M)
\cong 
 \begin{cases}
	   \SS{\a}\*_{\C} H^{\semiinf+0}(L\finn,M)
&\text{for }i=0,\\
0&\text{for }i\ne 0.
	  \end{cases}
\label{eq:characterization1}
\end{align}
This is also  an
isomorphism of
$\a$-modules
since
$\SS{\a}\cong \gr\SS{\a}$ as
{\em left} $\a$-modules,
where
$x_{\alpha}t^n\in \a$ is considered 
as an operator on
$\gr\SS{\a}=\bigoplus_p
F^p\SS{\a}/F^{p+1}\SS{\a}$
which 
maps $F^p\SS{\a}/F^{p+1}\SS{\a}$
to $F^{p+\alpha(\srhoche)}\SS{\a}/F^{p+\alpha(\srhoche)+1}\SS{\a}$.
We have computed the $E_1$-term \eqref{eq:E1-temp}:
\begin{align*}
 E_1^{\bullet,q}\cong \begin{cases}
		      \SS{\a}\*_{\C} H^{\semiinf+0}(L\finn,M)\*_{\C}
\bw{\bullet}(\finh[t\inv]t\inv)&\text{for }q=0,\\
0&\text{for }q\ne 0.
		      \end{cases}
\end{align*}
It follows that 
\begin{align}
 E_2^{p,q}\cong \begin{cases}
		 \SS{\a}\*_{\C}H^{\semiinf+0}(\a, M)&\text{for }p=q=0,\\0&\text{otherwise}
		\end{cases}
\label{eq:E_2-1}
\end{align}
as $\affh$-modules 
and $\a$-modules,
see the proof of Lemma \ref{Lem:characterization1}.
The spectral sequence
collapses at $E_2=E_{\infty}$ and we obtain the required isomorphism.
\end{proof}
Set
\begin{align*}
 Q_{\semiinf,+}=\sum_{\alpha\in \Delta^{\on{re}}\atop\bar \alpha\in
 \sDelta_-}\Z_{\geq 0}\alpha
+\Z_{\geq 0}\delta\subset\dual{\affh},
\end{align*}
and 
define 
the partial ordering $\leq_{\semiinf}$ on $\affh^*$
by 
$\mu\leq_{\semiinf} \lam$ $\iff$ $\lam-\mu\in Q_{\semiinf,+}$.
Note that
$\mu\leq_{\semiinf}\lam$ if and only if 
$t_{\alpha}\circ \mu\leq t_{\alpha}\circ \lam$ for  a sufficiently large 
$\alpha\in \sQche$.

\begin{proof}[Theorem \ref{Th:uniqueness-of-Wakimoto}]
Since The direction (i) $\Rightarrow$ (ii)
 in 
 Theorem \ref{Th:uniqueness-of-Wakimoto}
is obvious by \eqref{eq:homological-property-of-Wakimoto},
we shall prove that (ii) implies (i).
Let
$\{\lam_1,\dots,\lam_r\}$ 
be the set of weights of $H^{\semiinf+0}(\a,M)$ with multiplicities
counted, so
 that
\begin{align}
 M\cong \bigoplus_{i=1}^r \SS{\a}\*_{\C} \C_{\lam_i}
\label{eq:M-as}
\end{align}
as $\a$-modules and $\affh$-modules
by Proposition \ref{Pro:characterization-2}.
We may assume that
if $\lam_i\leq_{\semiinf}\lam_j$ then $j<i$.

Set $\lam=\lam_1$.
We shall show that there is a $\affg$-module embedding
$W(\lam)\hookrightarrow M$.
Let $\{\gamma_1,\gamma_2,\dots\}$ be a sequence
in $\sPchep$ such that
$\gamma_i-\gamma_{i-1}\in \sPchep$
and 
$\lim\limits_{n\ra \infty}\alpha(\gamma_n)=\infty$ for all $\alpha \in
\sroots_+$,
so that
$W(\lam)=\lim\limits_{\longrightarrow \atop n}M^{-\gamma_n}(\lam)$
by Proposition \ref{Pro:Wakimoto-as-limit}.
By Lemma \ref{Lem:twistinf-functor-TG=1 cases} (ii) we have
$M\cong T_{-\gamma_i}G_{-\gamma_i}(M)$,
and hence,
\begin{align*}
 \Hom_{\affg}(M^{-\gamma_i}(\lam), M) 
\cong \Hom_{\affg}(M(t_{\gamma_i}\circ \lam), G_{-\gamma_i}(M)).
\end{align*}
By \eqref{eq:M-as},
$\ch G_{-\gamma_i}(M)=\sum_{i=1}^r \ch M(t_{\gamma_i}\circ \lam)$.
Let $i$ be sufficiently large so that
 $t_{\gamma_i}\circ \lam$ is maximal in $G_{-\gamma_i}(M)$.
Denote by  $\Phi_i$
 the 
 $\affg$-module homomorphism
 $\psi_i:
M(t_{\gamma_i}\circ \lam)\ra G_{{-\gamma_i}}(M)$
which sends
$v_{t_{\gamma_i}\circ \lam}$ to 
a vector of $G_{-\gamma_i}(M)$ 
of weight $t_{\gamma_i}\circ \lam$.
As in the proof of Proposition \ref{Pro:Wakimoto-as-limit}
$\{T_{{-\gamma_i}}(\psi_i)
:M^{-\gamma_i}(\lam)\mapsto M\}$
 yield an injective  $\affg$-module homomorphism 
\begin{align*}
\Phi: W(\lam)=\lim\limits_{\longrightarrow \atop i}M^{-\gamma_i}(\lam)
\hookrightarrow  M.
\end{align*}
The  map $\Phi$
induces the
homomorphism
$H^{\semiinf+0}(\a, W(\lam))=\C_{\lam}\ra H^{\semiinf+0}(\a, M)$ 
which is certainly
 injective.
It follows from
the long exact sequence associated with the exact sequence
$0\ra W(\lam)\overset{\Phi}{\ra}M\ra M/W(\lam)\ra 0$ 
we obtain  that
 $H^{\semiinf+i}(\a,M/W(\lam))=0$ for $i\ne 0$
and $\dim H^{\semiinf+0}(\a,M/W(\lam))=\dim H^{\semiinf+0}(\a, M)-1$.
Theorem \ref{Th:uniqueness-of-Wakimoto}
 follows by the induction on
$\dim H^{\semiinf+0}(\a, M)$.
\end{proof}


\subsection{Twisted Wakimoto modules}
For $w\in \finW$
we have the decomposition
$\affg=w(\a)\+ w(\bar \a)$,
and $2\rho$ defines a semi-infinite $1$-cochain of 
the graded subalgebra $w(\bar \a)$.
Hence we can define
the {\em twisted Wakimoto module} $W^w(\lam)$ with highest weight $\lam$
and  twist $w\in \finW$
by
\begin{align*}
 W^w(\lam)=\on{S-ind}^{\g}_{w(\bar \a)}\C_{\lam},
\end{align*}
where $\C_{\lam}$ is the one-dimensional representation
of  $\affh$ corresponding to $\lam$
 regarded as a $\bar \a$-module by the projection $\bar \a\ra \affh$.
We have
\begin{align*}
& W^w(\lam) \cong \SS{w(\a)} \text{ as $w(\a)$-modules and }
\ch W^w(\lam)=\ch M(\lam),\\
&H^{\semiinf+i}(w(\a), W^w(\lam))\cong \begin{cases}
					\C_{\lam}& \text{for }i=0,\\0&\text{otherwise,}
				       \end{cases}
\text{ as $\affh$-modules.}
\end{align*}

Let $\{\gamma_1,\gamma_2,\dots\}$ be a sequence
in $\sPchep$ such that
$\gamma_i-\gamma_{i-1}\in \sPchep$
and 
$\lim\limits_{n\ra \infty}\alpha(\gamma_n)=\infty$ for all $\alpha \in
\sroots_+$.
The following assertion can be proved  in the same manner as Proposition 
\ref{Pro:Wakimoto-as-limit}.
\begin{Pro}
\label{Pro:twisted-Wakimoto-as-limit}
Let $\lam\in \dual{\affh}$,
$w\in \finW$.
There is an isomorphism of $\affg$-modules
\begin{align*}
W^w(\lam)\cong \lim\limits_{\longrightarrow\atop n}M^{-w(\gamma_n)}(\lam).
\end{align*}
 \end{Pro}

The following assertion can be proved  in the same manner as Theorem
\ref{Th:uniqueness-of-Wakimoto}.
\begin{Th}\label{Th:uniqueness-of-Wakimoto-2}
Let $\lam\in \dual{\affh}$  be non-critical,
$w\in \finW$.
 Let $M$ be an object of $\BGG^{\g}$ such that
\begin{align*}
H^{\semiinf+i}(w(\a),M)\cong \begin{cases}
		      \C_{\lam}& \text{if }i=0,\\
0&\text{otherwise},
		     \end{cases}
\end{align*}
as $\affh$-modules.
Then $M$ is isomorphic to $\Wakimoto^w(\lam)$.
\end{Th}

\section{Borel-Weil-Bott vanishing property of Twisting functors}
\label{section;twisting-functors}
\subsection{Left derived functors of twisting functors}
The functor
$T_w$,
$w\in \eW$, admits the left derived functor $\mc{L}_{\bullet}T_w$
in the category $\BGG^{\g}$ since it is a Lie algebra homology functor:
\begin{align*}
\mc{L}_i T_w(M)=\phi_w(H_i(\affg, S_w\*_{\C}M)),
\end{align*}
where 
$\g$ acts on
$N_w^*\*_{\C}M$ by 
$X(f\* m)=-fX\* m+f\* Xm$.
Because
\begin{align}
 \mc{L}_iT_w(M)\cong
\phi_w(H_i(\affn_w, N_w^*\*_{\C}M))
\label{eq:left-derived-functor-of-twitstin-functors}
\end{align}
as 
$w(\affn_w)$-modules,
  we have the following assertion.
\begin{Lem}\label{Lem:Li=0 if free}
Suppose  $M\in \BGG^{\g}$
 is free over $\affn_w$.
Then $\mc{L}_iT_w(M)=0$ for $i\geq 1$.
\end{Lem}

Let $\{e_i,h_i,f_i; i\in I\}$,
$e_i\in \affg_{\alpha_i}$,
$f_i\in \affg_{-\alpha_i}$,  be the Chevalley generators
of $\affg$.
For $i\in I$,
let
$\mf{sl}_2^{(i)}$
denote
the copy of $\mf{sl}_2$
 in $\affg$ 
spanned  by $\{e_i,h_i,f_i\}$
\begin{Pro}\label{Pro:L1Ts}
Let $M\in \BGG^{\g}$,
$i\in I$.
Denote by
$N$
the largest
$\mf{sl}_2^{(i)}$-integrable submodule of $M$.
Then  
$T_i(M)\cong T_i(M/N)$,
$\ch \mc{L}_1 T_i(M)\cong \ch N$
and
$\mc{L}_pT_i(M)=0$ for $p\geq 2$.
\end{Pro}
\begin{proof}
 Let
$T^{{(i)}}_{i}$ 
denote
the twisting functor for
$\mf{sl}_2^{(i)}$
corresponding to the reflection
$s_{\alpha_i}$.
Because
$T_i(M)\cong T_i^{(i)}(M)$
as $\mf{sl}_2^{(i)}$-modules and $\affh$-modules,
we have
\begin{align}
\mc{L}_p  T_{i}(M)\cong 
\mc{L}_p T^{(i)}_{i}(M)
\quad\text{as $\mf{sl}_2^{(i)}$-modules
and $\mf{\affh}$-modules.
}
\label{eq:T_i=T_i^i}
\end{align}
In particular
$\mc{L}_pT_i(M)=0$ for $p\geq 2$.
It follows that
the
exact sequence
\begin{align*}
 0\ra N\ra M\ra M/N\ra 0
\end{align*}
yields
 the 
 long exact sequence
\begin{align*}
 0\ra \mc{L}_1T_i(N)\ra \mc{L}_1T_i(M)\ra
\mc{L}_1T_i(M/N)\\\ra 
T_i (N)\ra T_i(M)\ra T_i(M/N)\ra 0.
\end{align*}
Since
$M/N$ is free as $\C[f_i]$-module 
$\mc{L}_1T_i(M/N)=0$
by Lemma \ref{Lem:Li=0 if free}.
Also,
$T_i(N)=0$
and $\mc{L}_1 T_i(N)\cong N$
as $\affh$-modules
by \cite[Theorem 6.1]{AndStr03}
and (\ref{eq:T_i=T_i^i}).
This completes the proof.
\end{proof}
Let $L(\lam)\in \BGG^{\affg}$ be the irreducible 
highest weight representation of
$\affg$ with highest weight $\lam\in \dual{\affh}$.
\begin{Th}[{\cite[Theorem 6.1]{AndStr03}}]
\label{Th:left-derived-functor-of-simple}
Let $\lam\in \dual{\affh}$
and
suppose that
$\bra \lam,\alpha_i\che\ket\in \Z_{\geq 0}$
with $i\in I$.
Then
\begin{align*}
\mc{ L}_p T_{i} (\Irr{\lam})
\cong \begin{cases}\Irr{\lam}&\text{if }p=1,\\
			  0&\text{if }p\ne 1.
			 \end{cases}
\end{align*}
\end{Th}
\begin{proof}
 The hypothesis implies that
$L(\lam)$ is $\mf{sl}_2^{(i)}$-integrable.
Therefore 
$\mc{ L}_p T_{i} (\Irr{\lam})=0$ for $p\ne 1$
and $\ch \mc{ L}_1 T_{i} (\Irr{\lam})=\ch L(\lam)$
by Proposition \ref{Pro:L1Ts}.
\end{proof}
\subsection{Twisting functors associated with integral Weyl group}
\begin{Lem}\label{Lem:simple-reflection-integral-Weyl}
Let $\lam\in \dual{\affh}$,
 $\alpha\in \Pi(\lam)$.
There exists $x\in \affW$
and $\alpha_i\in \Pi$
such that 
$s_{\alpha}=xs_{i}x\inv$,
$\ell(s_{\alpha})=2\ell(x)+1$
and
$\affrp\cap x (\affrm)\cap \Delta(\lam)=\emptyset$.
\end{Lem}
\begin{proof}
 Let $s_{\alpha}=s_{j_l}s_{j_{l-1}}\dots s_{j_1}$
be a 
reduced expression of $s_{\alpha}$ in $\affW$.
Then
\begin{align*}
\affrp\cap s_{\alpha}(\affrm)
=\{\alpha_1,s_{j_1}(\alpha_{j_2}),\dots,
s_{j_1}\dots s_{j_{l-1}}(\alpha_{j_l})\}
\end{align*}Since 
$\ell_{\lam}(\alpha)=1$,
$\affrp\cap s_{\alpha}(\affrm)\cap \Delta(\lam)=\{\alpha\}$.
Thus there exists 
$r$ such that
$\alpha=s_{j_1}\dots s_{j_{r-1}}
(\alpha_{j_r})$.
Set $x=s_{j_1}\dots s_{j_{r-1}}$,
$i=j_r$.
Then
$s_{\alpha}=s_{x(\alpha_{i})}=xs_{i}x\inv$.
It follows that
$s_{j_l}\dots s_{j_{r+1}}=x$ and 
$\ell(s_{\alpha})=2\ell(x)+1$.
Also
$\affrp\cap s_{\alpha}(\affrm)\cap \Delta(\lam)=\{\alpha\}$
implies that
$\affrp\cap x(\affrm)\cap \Delta(\lam)=\emptyset$.
\end{proof}
Note that if 
 $\lam$, $\alpha$,
$\alpha_i$, $x$ are  as in Lemma
  \ref{Lem:simple-reflection-integral-Weyl}
then 
\begin{align*}
 T_\alpha=T_x\circ T_i\circ T_{x\inv}.
\end{align*}

Let $\BGG_{[\lam]}^\g$ be the block of $\BGG^{\g}$
corresponding to $\lam$,
that is,
the full subcategory of $\BGG^{\g}$
consisting of objects $M$ such that
$[M: L(\mu)]
\ne 0\Rightarrow \mu\in \affW(\lam)\circ \mu$,
where 
$[M: L(\mu)]$ is the multiplicity of $L(\mu)$
in the local composition factor of $M$.

 \begin{Lem}\label{Lem:generic-twisting}
Let $\lam\in \dual{\affh}$,
$y\in \affW$, and suppose that 
$\bra \lam+\rho,\alpha\che\ket\not \in \Z$ for all 
$\alpha\in \affrp\cap y\inv(\affrm)$.
Then
$T_yM(w\circ \lam)\cong M(yw\circ \lam)$,
$ T_y L(w\circ \lam)\cong L(yw\circ \lam)$
for $w\in \affW(\lam)$.
Moreover
$T_w$ gives an equivalence of categories
$\BGG_{[\lam]}^\g\isomap \BGG_{[w\circ \lam]}^\g$.
The same is true for $G_w$.
 \end{Lem}
\begin{proof}
First note that
the assumption implies that 
$\affW(y\circ \lam)=y\affW(\lam)y\inv$.

We prove by induction on $\ell(y)$.
Let
$\ell(y)=1$,
so that $y=s_i$ for 
$i\in I$.
Then 
the fact that
$T_iM(w\lam)\cong M(s_iw\circ \lam)$
with $w\in \affW(\lam)$
follow from 
(\ref{eq:twisted-Verma=Verma}).
By \cite[Theorems 3.1, 3.2]{Ara04}
any object of $\BGG_{[\lam]}^{\g}$
and $\BGG_{[s_i\circ \lam]}^{\g}$ is free over
$\C[f_i]$ and cofree over $\C[e_i]$.
Hence
by Lemma \ref{Lem:twistinf-functor-TG=1 cases}
$T_i$ gives an equivalence of categories
$\BGG_{[\lam]}^{\g}\isomap \BGG_{[s_i\circ \lam]}^{\g}$
with a quasi-inverse $G_i$.
It follows that
$T_i L(\lam)$ is a simple $\affg$-module
which is a quotient of $T_i M(\lam)=M(s_i\circ \lam)$,
and hence is isomorphic to $L(s_i\circ \lam)$.
Next let $y=s_iz$ with $z\in \affW$,
$\ell(y)=\ell(z)+1$.
Then
$\affrp\cap y\inv(\affrm)=\{z\inv(\alpha_i)\}\sqcup
(\affrp\cap z\inv \affrm)$.
The assertion follows from the induction hypothesis.
\end{proof}

 \begin{Co}\label{Co:exact-functors}
Let $\lam$, $\alpha$,
$\alpha_i$, $x$ be as in Lemma
  \ref{Lem:simple-reflection-integral-Weyl}.
Then
$T_x$ give an equivalence of categories 
$\BGG_{[x \inv \circ \lam]}^\g\isomap
\BGG_{[\lam]}^{\g}$ such that
$T_x M(\mu)\cong M(x\circ \mu)$,
$T_x L(\mu)\cong M(x\circ \mu)$ 
for $\mu \in \affW(x\inv\circ \lam)\circ x\inv \lam=
x\inv \affW(\lam)\circ \lam$.
 \end{Co}

\begin{Lem}\label{Lem:T_i M^w}
 Let $\lam\in \dual{\affh}$,
$\alpha_i\in\Pi$ such that
$\bra \lam+\rho,\alpha_i\che\ket\not \in \Z$.
Then
  $T_i M^w(\lam)\cong M^{s_i ws_i}(s_i\circ \lam)$
for $w\in \affW(\lam)$.
\end{Lem}
\begin{proof}
By Lemma \ref{Lem:generic-twisting},
 $T_i M^w(\lam)\cong T_i T_w M(w\inv \circ \lam)
\cong T_iT_w T_i M(s_iw\inv \circ \lam)
\cong T^{s_i ws_i}M(s_iw\inv s_i s_i\circ \lam)$
.
\end{proof}

 \begin{Lem}\label{Lem:TT-generic}
 Let $\lam\in \dual{\affh}$,
$\alpha_i\in\Pi$ such that
$\bra \lam+\rho,\alpha_i\che\ket\not \in \Z$.
Then
$T_i^2: \BGG_{[\lam]}^\g\ra \BGG_{[\lam]}^\g$
 is isomorphic to the
  identity functor,
and so is
 $G_i^2: \BGG_{[\lam]}^\g\ra \BGG_{[\lam]}^\g$.
 \end{Lem}
 \begin{proof}
By Lemma \ref{Lem:generic-twisting}
$T_i^2$ induces an  auto-equivalence of the category
$\BGG_{[\lam]}^{\g}$
such that
 $T_i^2 M(w\circ \lam)\cong M(w\circ \lam)$
and
  $T_i^2(L(w\circ \lam))\cong L(w\circ \lam)$
for all $w\in \affW(\lam)$.
The standard argument shows that
such a functor must be isomorphic to the  identify functor.
 \end{proof}

\begin{Co}\label{Co:twisting-functor-for-integral-Weyl-group}
 Let $\lam\in \dual{\affh}$, 
$w=s_{\alpha}y\in \affW(\lam)$,
$\alpha\in \Pi(\lam)$,
$y\in \affW(\lam)$,
$\ell_{\lam}(w)=\ell_\lam(y)+1$.
Then
$T_w:\BGG_{[\lam]}^{\g}\ra \BGG_{[w\circ \lam]}^\g$
is isomorphic to the
functor $T_{s_{\alpha}}\circ T_y:
\BGG_{[\lam]}^{\g}\ra \BGG_{[w\circ \lam]}^{\g}$.
\end{Co}

\begin{Pro}\label{Pro:exact-sequeence}
Let $\lam\in \dual{\affh}$,
$w\in \affW(\lam)$,
$\alpha\in \Pi(\lam)$
and suppose that
$\bra w(\lam+\rho),\alpha\che\ket\not \in \N$.
Then the following sequence is exact:
 \begin{align*}
 0\ra M(s_{\alpha}w\circ \lam )\overset{\varphi_1}{\ra}
 M(w\circ \lam)
\overset{\varphi_2}{\ra} M^{s_{\alpha}}(w\circ \lam)
\overset{\varphi_3}{\ra} M^{s_{\alpha}}(s_{\alpha}w\circ \lam)\ra 0,
\end{align*}
where $\varphi_1,\varphi_2,\varphi_3$
are any non-trivial $\affg$-homomorphisms.
\end{Pro}
\begin{proof}
First observe that
$\Hom_{\affg}(M(s_{\alpha}w\circ \lam ),
 M(w\circ \lam))$,
$\Hom_{\affg}(
 M(w\circ \lam),M^{s_{\alpha}}(w\circ \lam))$
and
$\Hom_{\affg}(
 M^{s_{\alpha}}(w\circ \lam),M^{s_{\alpha}}(s_{\alpha}w\circ \lam))$
are all one-dimensional.
(The first and the third are one-dimensional by Theorem
 \ref{Th:hom-spaces-twiste-Verma}.)
By Lemma  \ref{Lem:simple-reflection-integral-Weyl}
there exists $x\in \affW$
and $\alpha_i\in \Pi$
such that 
$s_{\alpha}=xs_ix\inv$,
$\ell(s_{\alpha})=2\ell(x)+1$,
and 
$\affrp\cap x(\affrm)\cap \Delta(\lam)=\emptyset$.
We have
\begin{align*}
& M(y\circ \lam)\cong T_x M(x\inv y\circ \lam)
,\\
&M^{s_{\alpha}}(y\circ \lam)=T_x T_i T_{x\inv}M(xs_ix\inv y\circ \lam)
\cong T_x T_i M(s_ix\inv y\circ \lam)\cong
T_x M^{s_i}(x\inv y\circ \lam)
\end{align*}
for $y\in \affW(\lam)$
by Lemma \ref{Lem:generic-twisting}.
Since
$\bra x\inv w(\lam+\rho),\alpha_i\che\ket
=\bra w(\lam+\rho),\alpha\che\ket\in \N$
there is an exact sequence
\begin{align*}
 0\ra M(s_i x\inv w\circ \lam )\ra M(x\inv w\circ \lam)
\ra M^{s_i}(x\inv w\circ \lam)
\ra M^{s_i}(s_{i}x\inv w\circ \lam)\ra 0
\end{align*}
by  \cite[Propostion 6.2]{AndLau03}.
The required exact sequence is obtained by
applying  the exact functor
$T_x: \BGG^{\g}_{[x\inv \circ \lam]}\ra \BGG^{\g}_{[\lam]}$ to the above.
\end{proof}
 \begin{Pro}\label{Pro:L1Ts-for-interal-weyl}
Let $\lam\in \dual{\affh}$,
$\alpha \in \Pi(\lam)$,
$M\in \BGG^{\g}_{[\lam]}$.
Take
$\alpha_i \in \Pi$,
$x\in \affW$
such that $\alpha=x(\alpha_i)$
and
$x\inv \Delta(\lam)_+\subset \Delta^{re}_+
$ as in Lemma \ref{Lem:simple-reflection-integral-Weyl}.
Let $N'$  be the largest $\mf{sl}_2^{(i)}$-integrable submodule
of $T_{x\inv}(M)$
and set
$N=T_x(N')\subset M$.
Then
$T_{\alpha}(M)\cong T_{s_\alpha}(M/N)$,
$\ch \mc{L}_1 T_{s_\alpha}(M)=\ch N$ and $\mc{L}_p T_{s_\alpha}(M)=0$ for
  $p\geq 2$.
 \end{Pro}
 \begin{proof}
We have  $T_\alpha=T_x T_i T_{x\inv }$ 
and
$T_{x\inv}:\BGG^{\g}_{
[\lam]}\ra \BGG^{\g}_{[x\inv \circ \lam]}$,
$T_{x}:\BGG^{\g}_{x\inv\circ
\lam}\ra \BGG^{\g}_{[\lam]}$ 
are exact functors by Corollary \ref{Co:exact-functors}.
Therefore
\begin{align}
 \mc{L}_p T_{s_{\alpha}}(M)
=T_x(\mc{L}_p T_{i}(T_{x\inv}M)).
\label{eq;derived-functor-integral-Weyl}
\end{align}
Hence  Proposition \ref{Pro:L1Ts}
gives that
\begin{align*}
&T_{s_{\alpha}}(M)
=T_xT_i T_{x\inv} (M)\cong 
T_x T_i (T_{x\inv}(M)/N')
\cong T_xT_iT_{x\inv}(M/N)=T_{s_{\alpha}}(M/N),\\
&\ch \mc{L}_1T_{s_{\alpha}}(M)
=\ch T_x T_{x\inv}(N)=\ch N,
\\
&\mc{L}_p T_{s_{\alpha}}(M)=0\quad\text{for $p\geq 0$.}
\end{align*}
This completes the proof.
 \end{proof}

\begin{Th}\label{Th:twisting-integral-Weyl-group}
 Let $\lam\in \dual{\affh}$ be  regular dominant weight,
$w\in \affW(\lam)$.
Then
\begin{align*}
 \mc{L}_pT_w(L(\lam))\cong \begin{cases}
			    L(\lam)&\text{if }p=\ell(w),\\
0& \text{otherwise.}
			   \end{cases}
\end{align*}

\end{Th}
\begin{proof}
Let $\alpha\in \Pi(\lam)$.
Since
$T_{x\inv} L(\lam)=L(x\inv\circ  \lam)$
and $\bra x\inv\circ \lam+\rho,\alpha_i\che\ket
=\bra \lam+\rho,\alpha\che\ket \in \N$,
$T_{x\inv} L(\lam)$
is $\mf{sl}_2^{(i)}$-integrable.
Thus,
\begin{align*}
\mc{L}_p T_i T_{x\inv}L(\lam)\cong \begin{cases}
				T_{x\inv}L(\lam)&\text{if }p=1,\\
0&\text{if }p\ne 0
			       \end{cases}
\end{align*}
by Theorem \ref{Th:left-derived-functor-of-simple}.
It follows from
\eqref{eq;derived-functor-integral-Weyl}
 that
\begin{align}
 \mc{L}_p T_{s_{\alpha}}(L(\lam))\cong \begin{cases}
			    L(\lam)&\text{if }p=1,\\
0& \text{otherwise.}
			   \end{cases}
\label{eq:to-apply-Andersen-Stroppel}
\end{align}
Finally 
the assertion  follows in the same manner  as
in \cite[Corollary 6.2]{AndStr03}
by  Corollary \ref{Co:twisting-functor-for-integral-Weyl-group}.
\end{proof}

\section{Two-sided BGG resolutions of  admissible representations}
\label{section:semi-infnite BGG resolution} 
\subsection{Admissible representations}
A weight $\lam\in \dual{\affh}$ is called {\em admissible} if
it is regular dominant
and 
\begin{align*}
\Q \Delta(\lam)=\Q {\Delta}^{re}.
\end{align*}
The irreducible representation $L(\lam)$
is called admissible if $\lam$
is admissible.
A complex number $k$ is called
an  {\em admissible number} for $\affg$
if the weight
$k\Lam_0$   is admissible. 

Let $r\che$ be the lacing number of $\fing$,
that is,
the maximal number of the edges of the Dynkin digram of $\fing$.
Also,
let $h$ be the Coxeter number of $\fing$.
\begin{Pro}[\cite{KacWak89,KacWak08}]
\label{Pro:admissible number}
A complex  number  $k$ is admissible if and only if 
 \begin{align}
  k+h\che=\frac{p}{q}
 \quad \text{with }p,q\in \N,\ (p,q)=1,\ p\geq 
\begin{cases}
h\che&\text{if }(r\che  , q)=1
,\\
h&\text{if }(r\che, q)=r\che.
\end{cases}
\label{eq;ad-number}
 \end{align}
\end{Pro}

A complex number $k$ of the form (\ref{eq;ad-number})
is called  an {\em admissible number with denominator $q$}.
For an 
an admissible number $k$
with
denominator $q$,
we have
\begin{align*}
 &\Delta(k\Lam_0)=\{\alpha+nq\delta;\alpha\in \Delta,\ n\in \Z\}
\cong \Delta^{re}
\text{ and }\affW(k\Lam_0)\cong \affW
\text{ if }
(r\che,q)=1,\\
 &\Delta(k\Lam_0)\che=\{\alpha\che+nq\delta;\alpha\in \Delta,\ n\in \Z\}
\cong {}^L\Delta^{re}
\text{ and }\affW(k\Lam_0)\cong {}^L\affW
\text{ if }
(r\che,q)=r\che,
\end{align*}
where
$\Delta(\lam)\che=\{\alpha\che; \alpha\in \Delta(\lam)\}$ and
${}^L\Delta^{re}$ and
${}^L\affW$ are the real root system and  the  Weyl group
of the
non-twisted
affine Kac-Moody algebra
${}^L\affg$
associated with the Langlands dual ${}^L\fing$ of $\fing$,
respectively.
Set
\begin{align*}
\dot{\alpha_0}=
\begin{cases}
 -\theta+q\delta&\text{if }(r\che,q)=1,
\\
-\theta_s+\frac{q}{r\che}\delta&\text{if }(r\che ,q)=r\che.
\end{cases}
\end{align*}
Then
$\Pi(k\Lam_0)=
 \{\alpha_1,\dots, \alpha_{\ell}, \dot{\alpha}_0\}$.
Put
$\dot{s}_0=s_{\dot\alpha_0}\in \affW(k\Lam_0)$,
so that
$\affW(k\Lam_0)=\bra s_1,\dots, s_{\ell},\dot{s}_0\ket$.

For an admissible number $k$
let $Pr_k^+$ be the set of admissible weights $\lam$
of level $k$ such that
$\lam(\alpha\che)\in \Z_{\geq 0}$ for all $\alpha\in \sroots_+$.
Then
$\{L(\lam); \lam\in Pr_k\}$ is the
set of 
irreducible admissible representations of level $k$
which are integrable over $\fing\subset \affg$.
We have
$\Delta(\lam)=\Delta(k\Lam_0)$
for $\lam\in Pr_k^+$.

For an admissible number $k$
denote by  $Pr_k$  the set of admissible weights
$\lam$ of level $k$ such that $\Delta(\lam)\cong \Delta(k\Lam_0)$ as
root systems.
Then \cite{KacWak89}
\begin{align}
 Pr_k&=\bigcup_{y\in \eW
\atop y({\Delta}(k\Lam_0)\subset 
\Delta^{re}_+}Pr_{k,y},
\quad Pr_{k,y}=y\circ Pr_k^+.
\end{align}
Note that
\begin{align}
 \affW(\lam)=y\affW(k\Lam_0)y\inv \quad\text{for }\lam
\in Pr_{k,y}.
\end{align}

For $\lam\in Pr_k$,
let
 $\ell^{\semiinf}_{\lam}(?)$
be
the semi-infinite length function
of the affine Weyl group $\affW(\lam)$.
The
semi-infinite 
Bruhat ordering $\preceq_{\lam,\semiinf}$
are also defined for $\affW(\lam)$.
We will use the symbol 
$w\rhd_{\lam,\semiinf} w'$
 to denote a covering in the twisted Bruhat order
$\succeq_{\lam,\semiinf}$.


\begin{Rem}
 The admissible weight $\lam\in Pr_k$ is called 
the {\em principal admissible weight} \cite{KacWak89}
if $\Delta(\lam)\cong \Delta^{re}$,
that is, if the denominator $q$ of $k$
is prime to $r\che$.
\end{Rem}
\subsection{Fiebig's equivalence
and BGG resolution of admissible representations}
The following theorem is the special case of a result
of Fiebig \cite[Theorem 11]{Fie06}.
\begin{Th}[\cite{Fie06}]\label{Th:fiebig}
$ $
Let $\lam$ be regular dominant.
Suppose that
there exists a symmetrizable   Kac-Moody algebra
$\affg'$
whose Weyl group
$\affW'$
is isomorphic to
$\affW(\lam)$.
Let $\lam'$ be 
an integral dominant weight of $\affg'$,
$\BGG_{[\lam']}^{\affg'}$
the block of  $\BGG^{\affg'}$
containing the irreducible highest weight
representation $L^{\affg'}(\lam')$
of $\affg$ with highest weight $\lam'$.
Then
there is an equivalence of categories
\begin{align*}
 \BGG_{[\lam]}^{\affg}\cong \BGG_{[\lam']}^{\affg'}
\end{align*}
which maps $M(w\circ \lam)$ 
and $L(w\circ \lam)$,
$w\in \affW(\lam)$,
to $M^{\affg'}(\phi(w) \circ \lam')$
and $L^{\affg}(\phi(w)\circ \lam')$,
respectively.
Here
$M^{\affg'}(\lam')$
is
the Verma module of $\affg'$ with highest weight $\lam'$
and $\phi:\affW(\lam)\isomap \affW'$
is the isomorphism.
 \end{Th}

Let $k$ be an admissible number with denominator $q$,
$\lam \in Pr_k$.
By Theorem \ref{Th:fiebig}
the block  $\BGG^{\g}_{[\lam]}$ is equivalent
to a block of 
the category $\BGG$ of 
$\affg$ or ${}^L\affg$ containing
an integrable representation.
In particular
the existence 
of a BGG resolution of an integrable representation
of an affine  Kac-Moody algebra 
\cite{GarLep76, RocWal82}
implies the existence of
a BGG resolution 
for $L(\lam)$:

\begin{Th}\label{Th:BGG}
Let $k$ be an admissible number,
 $\lam\in Pr_k$.
Then there exists a complex
\begin{align*}
 \B_{\bullet}(\lam):
 \cdots \overset{d_{3}}{\rightarrow} \B_{2}(\lam)
\overset{d_{2}}{\rightarrow} \B_{1}(\lam)
\overset{d_{1}}{\rightarrow} \B_0(\lam)
\overset{d_0}{\rightarrow} 0
\end{align*}
of the form $\B_i(\lam)=\bigoplus\limits_{w\in \affW(\lam)
\atop \ell_{\lam}(w)=i}M(w\circ \lam)$,
$d_i=\sum\limits_{w,w'\in \affW(\lam)
\atop \ell_{\lam}(w)=i,\ w\rhd_{\lam} w'}d_{w',w}$,
$d_{w',w}\in \Hom_{\affg}(M(w\circ \lam), M(w'\circ \lam))$,
such that
\begin{align*}
 H_{i}(\B_{\bullet}(\lam))\cong 
\begin{cases}
 L(\lam)&\text{if }i=0,\\
0&\text{otherwise.}
\end{cases}
\end{align*}
\end{Th}

The resolution of $L(\lam)$ in Theorem \ref{Th:BGG}
can be combinatorially constructed as follows \cite{BerGelGel75}:
Fix a $\affg$-homomorphisms
\begin{align*}
i_{w',w}^\lam:M(w\circ \lam)\ra M(w'\circ \lam)
\end{align*}
for $w,w'\in \affW(\lam)$
with $w\succeq_{\lam}  w'$
in such a way that
$i_{w'',w'}^{\lam}\circ i^{\lam}_{w',w}=i^{\lam}_{w'',w}$
if  $w\succeq_{\lam} w'\succeq_{\lam} w$.

A quadruple $(w_1,w_2,w_3,w_4)$ in $\affW(\lam)$ is called a {\em square}
if $w_1\rhd_{\lam} w_2\rhd_{\lam} w_4$,
 $w_1\rhd_{\lam} w_3\rhd_{\lam} w_4$
and $w_2\ne w_3$.
\begin{Th}\label{Th:combinatorial-description-of-BGG}
Let $k$ be an admissible number,
 $\lam\in Pr_k$.
Assign 
$\ep_{w_2,w_1}\in \C^*$ for every 
pair $(w_1,w_2)$ in
$\affW(\lam)$ with  $w_1\rhd_{\lam}  w_2$  
in such a way that
$\ep_{w_4,w_2}\ep_{w_2,w_1}+\ep_{w_4,w_3}\ep_{w_3,w_1}=0$
for every
square $(w_1,w_2,w_3,w_4)$ of $\affW(\lam)$
(such an assignment is possible by \cite{BerGelGel75}).
Set
$d_{w',w}=\ep_{w',w}i_{w',w}^{\lam}$,
$d_i=\sum\limits_{w,w'\in \affW(\lam)
\atop 
\ell_{\lam}(w)=i,\ w\rhd_{\lam} w'}d_{w',w}$.
Then
\begin{align*}
 \B_{\bullet}(\lam):
 \cdots \overset{d_{3}}{\rightarrow}\B_{2}(\lam)
\overset{d_{2}}{\rightarrow} \B_{1}(\lam)
\overset{d_{1}}{\rightarrow} \B_0(\lam)
\overset{d_0}{\rightarrow} 0,
\end{align*}
where $\B_i(\lam)=\bigoplus\limits_{w\in\affW(\lam)
\atop \ell_{\lam}(w)=i}M(w\circ \lam)$,
is a resolution of $L(\lam)$.
\end{Th}

\subsection{Twisted BGG resolution}
For 
$w_1,w_2,y\in \affW(\lam)$
with $w_1\succeq_y w_2$,
set
\begin{align*}
\varphi^{\lam,y}_{w_2,w_1}=T_y(i_{y\inv w_2,y\inv w_1}^{\lam})
:M^y(w_1\circ \lam)\ra M^y(w_2\circ \lam).
\end{align*}

A quadruple $(w_1,w_2,w_3,w_4)$ in $\affW(\lam)$ is called a {\em
$y$-twisted square} 
if $w_1\rhd_y w_2\rhd_y w_4$,
 $w_1\rhd_y w_3\rhd_y w_4$
and $w_2\ne w_3$.
\begin{Th}\label{Th:twisted-BGG}
Let $k$ be an admissible number,
 $\lam\in Pr_k$,
$y\in \affW(\lam)$.
 Assign 
$\ep^y_{w_2,w_1}\in \C^*$ for every 
pair $(w_1,w_2)$
with 
$w_1\rhd_{\lam,y} w_2$  in  $\affW(\lam)$
in such a way that
$\ep^y_{w_4,w_2}\ep^y_{w_2,w_1}+\ep^y_{w_4,w_3}\ep^y_{w_3,w_1}=0$
for every
$y$-twisted square $(w_1,w_2,w_3,w_4)$ of $\affW(\lam)$.
Set
$\B_i^y(\lam)=\bigoplus\limits_{w\in \affW(\lam)
\atop
\ell_{\lam}^y(w)=i}M^y(w\circ \lam)$,
$d^y_{w',w}=\ep^y_{w',w}\varphi^{\lam,y}_{w',w}$,
$d_i=\sum\limits_{w,w'\in \affW(\lam)
\atop 
\ell^y_{\lam}(w)=i,\ w\rhd_{\lam,y} w'}d_{w',w}:
\B_i^y(\lam)\ra \B_{i-1}^y(\lam)$.
Then
\begin{align*}
 \B_{\bullet}^y(\lam):
 \cdots \overset{d_{3}}{\rightarrow} \B_{2}^y(\lam)
\overset{d_{2}}{\rightarrow} \B_{1}^y(\lam)
\overset{d_{1}}{\rightarrow} \B_0^y(\lam)
\overset{d_0}{\rightarrow} \B_{-1}^y(\lam)\ra \dots \ra \B_{-\ell(y)}^y(\lam)\ra 0
\end{align*}
is a complex of $\affg$-modules
such that
\begin{align*}
 H_i(B_\bullet^y(\lam))\cong
\begin{cases}
 L(\lam)&\text{for }i=0,\\
0&\text{otherwise.}
\end{cases}
\end{align*}
\end{Th}
\begin{proof}
 Set
$\ep_{y\inv w_1, y\inv w_2}
=\ep^y_{w_1,w_2}$.
Then 
$\{\ep^y_{w_1,w_2}\}$ satisfies the condition in Theorem if and only if
$\{\ep_{y\inv w_1, y\inv w_2}\}$ satisfies the condition in 
Theorem \ref{Th:BGG}.
In particular such an assignment is possible.
Consider the
BGG resolution  $\B_{\bullet}(\lam)$
of $L(\lam)$ in Theorem \ref{Th:combinatorial-description-of-BGG}
 associated with  this assignment.
We have
$\B_{\bullet}^y(\lam)=T_y (B_{\bullet}(\lam))[-\ell(y)]$,
where 
$[-\ell(y)]$ denotes the shift of the
degree.
Therefore the assertion follows  
from
Theorem \ref{Th:twisting-integral-Weyl-group}.
\end{proof}

\subsection{System of twisted BGG resolutions}
\begin{Pro}\label{Pro:compatibility-diagram}
Let $\lam\in \dual{\affh}$ be regular dominant,
 $y=s_{\beta_i}s_{\beta_2}\dots s_{\beta_l}$
 a  reduced expression of $y\in \affW(\lam)$
with $\beta_i\in \Pi(\lam)$.
Set
$y_i=s_{\beta_1}s_{\beta_2}\dots s_{\beta_i}$
for $i=0,1 ,\dots, l$
and
fix a non-zero $\affg$-homomorphism
$\phi_w^{y_i}: M^{y_i}(w\circ \lam)\ra M^{y_{i+1}}(w\circ \lam)$
for $w\in \affW(\lam)$,
$i=1,\dots, l$.
One can assign
$\epsilon^{i}_{w_2,w_1}\in \C^*$
for each pair $(w_1,w_2)$ with $w_1\rhd_{\lam,y_i}w_2$
for all $i=1,\dots, l$
in such a way that
the following hold: 
\begin{enumerate}
\item
$\ep^i_{w_4,w_2}\ep^i_{w_2,w_1}+\ep^i_{w_4,w_3}\ep^i_{w_3,w_1}=0$
for every
$y_i$-twisted square $(w_1,w_2,w_3,w_4)$
of $\affW(\lam)$,
 \item
If
$w_1\rhd_{\lam,y_i}w_2$,
$w_1\rhd_{\lam,y_{i-1}}w_2$,
$\ell^{y_i}_{\lam}(w_1)=\ell^{y_{i-1}}_{\lam}(w_1)$
and $\ell^{y_i}_{\lam}(w_2)=\ell^{y_{i-1}}_{\lam}(w_2)$,
then
the the following diagram commutes.
\begin{align}
  \begin{CD}
M^{y_{i-1}}(w_1\circ \lam)
   @>\ep^{i-1}_{w_2,w_1}\varphi^{\lam,y_{i-1}}_{w_2,w_1}
 >> M^{y_{i-1}}(w_2\circ \lam)\\
@V\phi_{w_1}^{y_{i-1}}
 VV @VV \phi_{w_2}^{y_{i-1}} V\\
M^{y}(w_1\circ \lam) @>
\ep^{i}_{w_2,w_1}\varphi^{\lam,y_i}_{w_2,w_1} >> M^{y}(w_2\circ \lam).
  \end{CD}
\label{eq:commuattive-diagram-in-Pro}
\end{align}

\end{enumerate}
\end{Pro}

\begin{Pro}\label{Pro:non-zero}
Let $\lam\in \dual{\affh}$ be regular dominant,
 $y\in \affW(\lam)$,
$\alpha\in \Pi(\lam)$
such that 
$\ell_{\lam}(ys_{\alpha})=\ell_{\lam}(y)+1$.
Set $\beta=y(\alpha)$
 \begin{enumerate}
  \item Let
$w_1,w_2\in \affW(\lam)$.
Suppose that
$w_1\rhd_{y}w_2$,
 $w_1\rhd_{ys_{\alpha}}w_2$
and
$\ell^y_{\lam}(w_1)=\ell^{ys_{\alpha}}_{\lam}(w_1)$.
Then
\begin{align*}
\dim_{\C} \Hom_{\g}(M^y(w_1\circ \lam), M^{ys_{\alpha}}(w_2\circ \lam))
=1.
\end{align*}
Moreover,
either of the followings span
the one-dimensional vector space
$\Hom_{\g}(M^y(w_1\circ \lam), M^{ys_{\alpha}}(w_2\circ \lam))$:
\begin{enumerate}
 \item the composition  $M^y(w_1
\circ \lam)\ra M^y(w_2\circ \lam)\ra
       M^{ys_{\alpha}}(w_2\circ \lam)$
of any non-trivial $\affg$-homomorphisms;
 \item the composition  $M^y(w_1
\circ \lam){\ra} M^{ys_{\alpha}}(w_1\circ \lam)
{\ra} 
M^{ys_{\alpha}}(w_2\circ \lam)$
of any non-trivial $\affg$-homomorphisms.
\end{enumerate}
\item 
Let $w_1,w_2\in \affW(\lam)$.
Suppose that
$\ell^y_{\lam}(w_1)=\ell^y_{\lam}(w_2)+2$
and
$w_i\inv (\beta)\in \affrp$ for $i=1,2$.
Then the composition
$M^y(w_1\circ \lam)
\ra M^y(w_2\circ \lam)\ra M^{ys_{\alpha}}(w_2\circ \lam)$
of any non-trivial homomorphisms
is non-zero.
\item
Let $w\in \affW(\lam)$
and suppose that
$s_{\alpha}w\rhd_{\lam,y}w$.
Then the composition 
$M^y(s_{\alpha}w\circ \lam)\ra M^y(w\circ \lam)\ra
     M^{ys_{\alpha}}(w\circ \lam)$
of any  $\affg$-homomorphisms
is zero.
 \end{enumerate}
\end{Pro}
\begin{proof}
(i) 
Since  $y\inv w_1\rhd y\inv w_2$,
the Jantzen sum formula implies that
\begin{align*}
[M(y\inv w_2\circ \lam): L(y\inv w_1\circ \lam)]=1.
\end{align*}
Hence 
$[M^{s_{\alpha}}(y\inv w_2\circ \lam): L(y\inv w_1\circ \lam)]=1$.
As
\begin{align*}
\Hom (M^y(w_1\circ \lam), M^{ys_{\alpha}}(w_2\circ \lam))
\cong \Hom (M(y\inv w_1\circ \lam), M^{s_{\alpha}}(y\inv w_2\circ
 \lam)),
\end{align*}
it follows that
\begin{align*}
 \dim_\C  \Hom (M^y(w_1\circ \lam), M^{ys_{\alpha}}(w_2\circ \lam))
\leq 1
\end{align*}

Now we have
\begin{align*}
\Hom_{\g}(M^y(w_1\circ \lam), M^y(w_2\circ \lam))
\cong \Hom_{\g}(M(y\inv w_1\circ \lam), M(y\inv w_2\circ \lam)),
\\
\Hom_{\g}(M^y(w_1\circ \lam), M^{ys_{\alpha}}(w_1\circ \lam))
\cong \Hom_{\g}(M(y\inv w_1\circ \lam), M^{s_{\alpha}}(y\inv w_1\circ \lam)),
\\\Hom_{\g}(M^y(w_2\circ \lam), M^{ys_{\alpha}}(w_2\circ \lam))
\cong \Hom_{\g}(M(y\inv w_2\circ \lam), M^{s_{\alpha}}(y\inv w_2\circ
 \lam)),
\\
\Hom_{\g}(M^{ys_{\alpha}}(w_1\circ \lam), M^{ys_{\alpha}}(w_2\circ \lam))
\cong \Hom_{\g}(M(s_{\alpha}y\inv w_1\circ \lam), M(s_{\alpha}y\inv
 w_2\circ \lam)).
\end{align*}
In particular they are all one-dimensional.
Hence it remains  to show that
the compositions in (a) and (b) are non-trivial.
This is equivalent to the non-triviality of the 
compositions
\begin{align*}
& M(y\inv w_1\circ \lam)\ra
M(y\inv w_2\circ \lam)\ra M^{s_{\alpha}}(y\inv w_2\circ \lam)\\
&\text{and }M(y\inv w_1\circ \lam)\ra
M^{s_{\alpha}}(y\inv w_1\circ \lam)\ra M^{s_{\alpha}}(y\inv w_2\circ \lam),
\end{align*}
respectively.
Therefore we may assume that $y=1$.

Since  $\bra w_2(\lam+\affrho),\alpha\che\ket \in \N$,
we have the exact sequence 
\begin{align}
 0\ra M(s_{\alpha}w_2\circ \lam )\ra M(w_2\circ \lam)
\ra M^{s_\alpha}(w_2\circ \lam)
\ra M^{s_\alpha}(s_{\alpha}w_2\circ \lam)\ra 0
\label{eq:kernek-of-M->sM}
\end{align}
 by Proposition 
\ref{Pro:exact-sequeence}.
On the other hand
\begin{align}
w_1\circ \lam\not \preceq_{\lam} s_\alpha w_2\circ \lam
\label{eq-not-lower}
\end{align}
as
we have the square
$(s_{\alpha}w_1,w_1,s_{\alpha}w_2,w_2)$
 by the assumption
and (\ref{eq:BGG-equivalence-y}).
Hence
\eqref{eq:kernek-of-M->sM} implies that
 the image of the highest weight vector of $M(w_1\circ \lam)$
in $M(w_2\circ \lam)$ does not lie    in the kernel of the map
$M(w_2\circ \lam)\ra M^{s_\alpha}(w_2\circ \lam)$.
This proves the non-triviality
of the composition map
in (a) for $y=1$, and thus, for all $y$.
Next 
we show the non-triviality of the composition
in (b).
Consider the exact sequence
\begin{align*}
0\ra M(s_\alpha w_1\circ \lam)\ra M(s_\alpha  w_2\circ \lam)
\ra N\ra 0
\end{align*}
in the category $\BGG_{[\lam]}^{\g}$,
where
$N=M(s_\alpha w_2\circ \lam)/M(s_\alpha w_1\circ \lam)$.
Applying the functor $T_{s_{\alpha}}$
we obtain the exact sequence 
\begin{align}
0\ra \mc{L}_1 T_{s_{\alpha}} N
\ra M^{s_\alpha}(w_1\circ \lam)
\ra M^{s_\alpha}(w_2\circ \lam)\ra T_i N\ra 0.
\label{eq:exact-uin-the-proof}
\end{align}
By Proposition \ref{Pro:L1Ts-for-interal-weyl},
the weights 
of $\mc{L}_1 T_{s_{\alpha}} N$
is contained in
the set of 
weights of $N$,
and hence of
 $M(s_\alpha w_2\circ \lam)$.
Therefore 
\eqref{eq-not-lower} and 
\eqref{eq:exact-uin-the-proof}
imply that
 the image of the highest weight vector 
of $M(w_1 \circ \lam)$ in $M^{s_\alpha}(w_1 \circ \lam)$
does not belong to the kernel of
the map $M^{s_{\alpha}}(w_1\circ \lam)\ra M^{s_\alpha}(w_2\circ \lam)$.
This competes the proof of (i).
(ii)
Similarly as above, the problem reduces to the case
$y=1$.
By the assumption
we have 
$s_{\beta}w_1\rhd_{\lam}w_1$,
$s_{\beta}w_2\rhd_{\lam}w_2$.
Thus
 $w_1\not \preceq_{\lam} s_\beta w_2$
because otherwise $(w_1,s_\beta w_1,s_\beta w_1,w_2)$
is a  square.
Hence 
\eqref{eq:kernek-of-M->sM}
proves the assertion by the same argument  as above.
(iii) Again we may assume that
$y=1$ and 
the assertion follows from (\ref{eq:kernek-of-M->sM}).
\end{proof}

\begin{proof}[Proof of Proposition \ref{Pro:compatibility-diagram}]
 We prove by induction on $i$ that
 such an assignment is possible.
 
As we already remarked
the case 
$i=0$ is the well-known result of \cite{BerGelGel75}.
So let $i>0$.
Suppose that
$w_1\rhd_{\lam,y_i}w_2$.
Set $\beta=y_{i-1}(\alpha_i)\in 
\affrp
$.
The following
four cases are possible.
(The case 
$w_1\inv(\beta)\in \affrp$,
$w_2\inv(\beta)\in \affrm$
does not happen by 
\cite[Lemma 11.3]{BerGelGel75}.)

I) $w_1\inv(\beta), w_2\inv (\beta)\in \affrp$.
In this case 
$w_1\rhd_{\lam,y_{i-1}} w_2$,
 $\ell^{y_i}_{\lam}(w_1)=\ell^{y_{i-1}}_{\lam}(w_1)$
and $\ell^{y_i}_{\lam}(w_2)=\ell^{y_{i-1}}_{\lam}(w_2)$.
By Proposition \ref{Pro:non-zero} 
there exists a unique $\ep^i_{w_2,w_1}$
which makes the digram  
\eqref{eq:commuattive-diagram-in-Pro}
commutes.

III)
$w_1=s_{\beta}w_2$.
In this case 
$w_2\rhd_{\lam,y_{i-1}} w_1$,
 $\ell^{y_i}_{\lam}(w_1)=\ell^{y_{i-1}}_{\lam}(w_1)-2$
and $\ell^{y_i}_{\lam}(w_2)=\ell^{y_{i-1}}_{\lam}(w_2)$.
We set
$\ep^i_{w_2,w_1}=\ep^{i-1}_{w_1,w_2}$.

III)
$w_1\inv(\beta), w_2\inv (\beta)\in \affrm$.
In this case 
$w_1\rhd_{\lam,y_{i-1}} w_2$,
 $\ell^{y_i}_{\lam}(w_1)=\ell^{y_{i-1}}_{\lam}(w_1)-2$
 $\ell^{y_i}_{\lam}(w_2)=\ell^{y_{i-1}}_{\lam}(w_2)-2$,
and we have the  $y_i$-twisted square
$(w_1,s_{\beta}w_1,w_2,s_{\beta}w_2)$.
Note that $\ep^i_{s_{\beta}w_2,s_{\beta}w_1}$
is defined in I),
and $\ep^i_{s_{\beta}w_1,w_1}$,
$\ep^{i}_{s_{\beta}w_2,w_2}$
are defined in II).
We set
\begin{align}
 \ep^i_{w_2,w_1}
=-\frac{\ep^i_{s_{\beta}w_1,w_1}\ep^i_{s_{\beta}w_2,s_{\beta}w_1}}{\ep^i_{s_{\beta}w_2,w_2}.}
\label{eq:def-epsilon-III}
\end{align}

IV) 
$w_1\inv(\beta)\in \affrm$,
$w_2\inv(\beta)\in \affrp$,
$w_2\ne s_{\beta} w_1$.
In this case
there exists a unique $w_3\in \affW$
such that
$(s_{\beta}w_1,w_1,w_3,w_2)$ is a $y_i$-twisted square.
Note that
$w_3\inv(\beta)\in \affrp$
because $(w_3,w_2,s_{\beta}w_3,s_{\beta}w_2)$
is a $y_i$-twisted square 
by (\ref{eq:BGG-equivalence-y}).
Since
$\ep^i_{w_3,s_{\beta}w_1}$,
$\ep^i_{w_2,w_3}$ are  defined in I)
and 
$\ep^i_{w_1,s_{\beta}w_1}$ is defined in II),
we can set
\begin{align}
 \ep^i_{w_1,w_1}=-\frac{
\ep^i_{w_3,s_{\beta}w_1}\ep^i_{w_2,w_3}}{\ep^i_{w_1,s_{\beta}w_1}}.
\label{eq:def-epsilon-IV}
\end{align}

Now
let
$(w_1,w_2,w_3,w_4)$ be a
$y_i$-twisted square.
Set
\begin{align*}
A_i(w_1,w_2,w_3,w_4)=\frac{\ep^i_{w_4,w_2}\ep^i_{w_2,w_1}}{\ep^i_{w_4,w_3}\ep^i_{w_3,w_1}}.
\end{align*}
We need to show that 
$A_i(w_1,w_2,w_3,w_4)=-1$.

The following four cases are possible.

1) $w_2=s_{\beta}w_1$, $w_4=s_{\beta}w_3$.
In this case 
the assertion follows from the
definition 
 \eqref{eq:def-epsilon-III}.

2) $w_2=s_{\beta}w_1$, $w_4\ne s_{\beta}w_3$.
In this case
$(s_{\beta}w)\inv(\beta)\in \affrm$,
and $w_4\inv(\beta)\in \affrp$
because otherwise
$w_3=s_{\beta}w_4$.
Hence the assertion follows from 
the definition (\ref{eq:def-epsilon-IV}).

3) $w_2\ne s_{\beta}w_1$, $w_4= s_{\beta}w_3$.
In this case
$(s_{\beta}w_1,w_1,s_{\beta}w_2.w_2)$,
$(s_{\beta}w_1,w_1,s_{\beta}w_2,w_3)$,
$(s_{\beta}w_2,w_2,s_3,w_4)$
are $y_i$-twisted squares:
\begin{align*}
 \xymatrix{s_{\beta}w_1\ar[r]_{y_i}
\ar[rd]_{y_i}& w_1\ar[rd]
\ar[r]_{y_i} & w_2\ar[rd]_{y_i} &
\\ 
 & s_{\beta}w_2\ar[ru]\ar[r]_{y_i} &w_3
\ar[r]_{y_i}& s_{\beta}w_3
}
\end{align*}
We have
by 1) 
\begin{align*}
A_i(s_{\beta}w_1,w_1,s_{\beta}w_2,w_2)=
A_i(s_{\beta}w_2,w_2,w_3,s_{\beta}w_3)=-1
\end{align*}
and by 2)
\begin{align*}
A_i(s_{\beta}w_1,w_1,s_{\beta}w_2,w_3)
=-1.
\end{align*}
But 
\begin{align*}
&A_i(w_1,w_2,w_3,s_{\beta}w_3)\\
&=A_i(s_{\beta}w_1,w_1,s_{\beta}w_2,w_2)
A_i(s_{\beta}w_2,w_2,w_3,s_{\beta}w_3)
A_i(s_{\beta}w_1,s_{\beta}w_2,w_1,w_3).
\end{align*}
Hence the assertion follows.

4)
$w_2\ne s_{\beta}w_1$, $w_4\ne  s_{\beta}w_2$.
we see as in \cite[p.57, c)]{BerGelGel75} that
$w_4\ne s_{\beta}w_2, s_{\beta}w_3$,
and  hence as in \cite[p.56, 1)]{BerGelGel75} we find that  
$(s_{\beta}w_1,s_{\beta}w_2,s_{\beta}w_3,s_{\beta}w_4)$
is also a $y_i$-twisted square.
Hence a) $w_i\inv (\beta)\in \affrp$
for all $i$ or 
b)
$w_i\inv (\beta)\in \affrm$
for all $i$.

a) The case 
 $w_i\inv (\beta)\in \affrp$
for all $i$:
By the definition I)
we have the commutative diagram
\begin{align}
  \begin{CD}
M^{y_{i-1}}(w_1\circ \lam)
   @>\ep^{i-1}_{w_4,w_a}\ep^{i-1}_{w_a,w_1}
\varphi^{\lam,y_{i-1}}_{w_4,w_1}
 >> M^{y_{i-1}}(w_4\circ \lam)\\
@V\phi_{w_1}^{y_{i-1}} VV @VV \phi_{w_4}^{y_{i-1}} V\\
M^{y}(w_1\circ \lam) @>
\ep^{i}_{w_4,w_a}\ep^{i}_{w_a,w_1}
\varphi^{\lam,y_{i}}_{w_4,w_1}
 >> M^{y}(w_4\circ \lam)
  \end{CD}
\end{align}
for $a=2,3$.
Since 
$\ep^{i-1}_{w_4,w_2}\ep^{i-1}_{w_2,w_1}=-
\ep^{i-1}_{w_4,w_3}\ep^{i-1}_{w_3,w_1}$
by the induction hypothesis
the commutativity of the above diagram 
implies that
$\ep^{i}_{w_4,w_2}\ep^{i}_{w_2,w_1}
=-\ep^{i}_{w_4,w_3}\ep^i_{w_3,w_1}$
by
Proposition \ref{Pro:non-zero} (ii).

b) The case
that 
$w_i\inv(\beta)\in \affrm$
for all $i$:
We have that  
$(s_{\beta}w_1,w_1,s_{\beta}w_2,w_2)$,
$(s_{\beta}w_1,w_1,s_{\beta}w_3,w_3)$,
$(s_{\beta}w_1,s_{\beta}w_2,s_{\beta}w_3,s_{\beta}w_4)$,
$(s_{\beta}w_2,w_2,s_{\beta}w_4,w_4)$
and $(s_{\beta}w_3,w_3,s_{\beta}w_4,w_4)$
are   all $y_i$-twisted squares.
Hence the assertion follows from
the 
equality 
\begin{align*}
A_i(w_1,w_2,w_3,w_4)A_i(s_{\beta}w_1,s_{\beta}w_2,w_1,w_2)
A_i(s_{\beta}w_1,w_1,s_{\beta}w_3,w_3)
\\=
A_i(s_{\beta}w_1,s_{\beta}w_2,s_{\beta}w_3,
s_{\beta}w_4)
A_i
(s_{\beta}w_2,w_2,s_{\beta}w_4,w_4)
A_i(s_{\beta}w_3,s_{\beta}w_4,w_3,w_4).
\end{align*}
\end{proof}

Let $k$ be an admissible number,
$\lam\in Pr_k$.
Let $y\in \affW(\lam)$,
 $\{y_i\}$,
$\{\phi^{y_i}_w\}$,
$\{\ep^i_{w_2,w_1}\}$
 be as in Proposition
 \ref{Pro:compatibility-diagram}.
Because 
$\{\ep^i_{w_2,w_1}\}$ satisfies the condition in Theorem 
\ref{Th:twisted-BGG}
there is a corresponding
twisted BGG resolution
$\B^{y_i}_{\bullet}(\lam)$
of $L(\lam)$
for $i=0,1,\dots, l=\ell(y)$.
Define
\begin{align*}
\Phi^{y_{i+1},y_i}_p
=\bigoplus_{w\in \affW(\lam)
\atop \ell^{y_i}_{\lam}(w)=\ell^{y_{i+1}}_{\lam}(w)=p} 
\phi^{y_{i+1},y_i}_w: 
\mc{B}^{y_i}_{p}(w\circ \lam)\ra  \mc{B}_{p}^{y_{i+1}}(w\circ \lam).
\end{align*}

\begin{Pro}\label{Pro:inducitive-system-of-resolutions}
In the above setting
$\Phi^{y_{i+1},y_i}_{\bullet}$
gives 
a quasi-isomorphism 
$\mc{B}^{y_i}_{\bullet}(\lam)\sim  \mc{B}_{\bullet}^{y_{i+1}}(\lam)$
of complexes
for each $i=0,1,\dots, l-1$.
\end{Pro}

\begin{Lem}\label{Lem: homo-of-complexes?}
 Let  $\lam\in \dual{\affh}$,
$y$,
$y_i$ be as in Proposition \ref{Pro:compatibility-diagram},
$w_1,w_2\in \affW(\lam)$.
\begin{enumerate}
 \item Suppose that
$w_1\rhd_{\lam,y_i} w_2$,
$\ell^{y_i}(w_1)=\ell^{y_{i+1}}(w_1)$.
Then
$w_1\rhd_{\lam,y_{i+1}}w_2$.
 \item Suppose that
$w_1\rhd_{\lam,y_i} w_2$,
$\ell^{y_i}(w_2)=\ell^{y_{i+1}}(w_2)$.
Then
either of the following two holds.
\begin{enumerate}
 \item $w_2=s_{\beta}w_1$ and 
$w_2\rhd_{\lam,y_{i+1}} w_1$.
\item  $w_1\rhd_{\lam,y_{i+1}} w_2$.
\end{enumerate}
\end{enumerate}
\end{Lem}
\begin{proof}
 (1) By assumption
$s_{\beta}w_1\rhd_{\lam,y_{i}} w_2$.
Therefore
$(s_{\beta}w_1,w_1,s_{\beta}w_2,w_2)$
 is a $y_i$-twisted square.
(2) Similarly,
if $w_2\ne s_{\beta}w_1$
then
$(s_{\beta}w_1,w_1,s_{\beta}w_2,w_2)$
$y_i$-twisted square.
The $w_2\ne s_{\beta}w_1$ case is obvious.
\end{proof}

\begin{proof}[Proof of Proposition \ref{Pro:inducitive-system-of-resolutions}]
The fact that $\Phi_{\bullet}^{y_i}$
defines a homomorphism of complexes
follows from
the commutativity of (\ref{eq:commuattive-diagram-in-Pro}),
Proposition \ref{Pro:non-zero} (iii),
and Lemma \ref{Lem: homo-of-complexes?}.
Since both complexes are quasi-isomorphic to
$L(\lam)$,
to show that it defines a quasi-isomorphism
it suffices to check that
it defines a non-trivial homomorphism
between the corresponding homology spaces.
This follows from the fact that
$\phi^{y_i}_1: M^{y_i}(\lam)\ra M^{y_{i+1}}(\lam)$
sends the highest weight vector of $M^{y_i}(\lam)$ to the highest weight 
vector of $M^{y_{i+1}}(\lam)$.
\end{proof}

\subsection{Two-sided BGG resolutions
of  $G$-integrable admissible representations}
For $\lam\in Pr_k$
set \begin{align*}
\affW^i(\lam)
=\{w\in \affW(\lam);
\ell^{\semiinf}_{\lam}(w)=i\}.
    \end{align*}
We note that
\begin{align*}
\sharp \affW^i(\lam)=\begin{cases}
		      1&\text{if }\fing=\mf{sl}_2,\\
\infty&\text{else.}
		     \end{cases}
\end{align*}
\begin{Th}\label{Th:two-sided BGG}
Let $k$ be an admissible number,
 $\lam\in Pr_k^+$
\begin{enumerate}
 \item 
The space
$ \Hom_{\g}(W(w\circ \lam), W(w'\circ \lam))$ is one-dimensional
for $w,w'\in \affW(\lam)$ such that $w\rhd_{\lam,\semiinf} w'$.
\item
There exists 
 a complex 
\begin{align*}
 C^{\bullet}(\lam): \cdots \ra
C^{-2}(\lam)
\overset{d_{-2}}{\rightarrow} C^{-1}(\lam)
\overset{d_{-1}}{\rightarrow} C^0(\lam)
\overset{d_0}{\rightarrow} C^1(\lam)
\overset{d_1}{\rightarrow} C^2(\lam)\overset{d_2}{\rightarrow} \cdots
\end{align*}
in the category $\BGG$
of the form
\begin{align*}
 &C^i(\lam)=\bigoplus\limits_{w\in \affW^i(\lam)}
W(w\circ \lam),
\quad
d_i=\sum\limits_{
w\in \affW^i(\lam),\ w'\in \affW^{i+1}(\lam)
\atop w\rhd_{\lam,\semiinf} w'}
d_{w',w},
\end{align*}
where
$d_{w',w}$
is a  non-trivial
 $\affg$-homomorphism $
W(w\circ \lam)\ra W(w'\circ \lam)$,
such that
\begin{align*}
 H^i(C^{\bullet}(\lam))\cong \begin{cases}
			      L(\lam)&\text{for } i=0,\\ 0&\text{for }i\ne 0.
			     \end{cases}
\end{align*}
\end{enumerate}
\end{Th}
\begin{proof}
(ii)
Let $q$ be the denominator  of $k$
and 
set $M=q\sQche$ if
$(r\che,q)=1$
and $M=q\overset{\circ}{Q}$ if $(r\che ,q)=r\che$,
so that $\affW(\lam)=\finW\ltimes t_M$.
Let 
$\gamma_1,\gamma_2,\dots,$ be a sequence in $\sPchep\cap M$ 
such that
$\gamma_i-\gamma_{}\in \sPchep\cap M$,
$\lim\limits_{i\ra \infty} \alpha(\gamma_i)=\infty$ for all
$\alpha\in \sDelta_+$.

By Proposition  \ref{Pro:inducitive-system-of-resolutions}
there is   an inductive system
$\{\mc{B}^{-\gamma_i}_{\bullet}(\lam)\}$
of twisted BGG resolutions.
Let
$\mc{B}^{\bullet}_{-\gamma_i}(\lam)
$ be the
complex $\mc{B}^{-\gamma_i}_{\bullet}(\lam)$ with the opposite
 homological
grading.
Thus
it is a complex 
\begin{align*}
 B^{\bullet}_{-\gamma_i}(\lam):
\ra 
 \cdots \overset{d_{-2}}{\rightarrow} \mc{B}^{-1}_{\gamma_i}(\lam)
\overset{d_{-1}}{\rightarrow} \mc{B}^{0}_{-\gamma_i}(\lam)
\overset{d_{0}}{\rightarrow} B^1_{\gamma_i}(\lam)
\overset{d_1}{\rightarrow} \ra \cdots 
\end{align*}
of the form
$\mc{B}_{\gamma_i}^p(\lam)=\bigoplus\limits_{w\in \affW(\lam)\atop
\ell_{\lam}^{-\gamma_i}(w)=-p}M^{-\gamma_i}(w\circ \lam)$,
$d_p=\sum\limits_{w,w'\in \affW(\lam)\atop
\ell_{\lam}^{-\gamma_i}(w)=-p,
w\rhd_{\lam,t_{-\gamma_i}}w' }d_{w',w}^{\gamma_i}
$,
$d_{w',w}^{\gamma_i}
: M^{-\gamma_i}(w\circ \lam)\ra M^{-\gamma_i}(w'\circ
 \lam)$
such that
$H^p(B^{\bullet}_{-\gamma_i}(\lam))
=\begin{cases}
  L(\lam)&\text{if }p=0,\\ 0&\text{otherwise.}
 \end{cases}$

Let
$(C^{\bullet}(\lam),d_\bullet)$ be the 
complex obtained as the inductive limit
of complex $\mc{B}^{\bullet}_{-\gamma_i}(\lam)$.
By Lemma \ref{Lem:semi-infnite-length},
Proposition \ref{Pro:Wakimoto-as-limit} 
and Proposition \ref{Pro:inducitive-system-of-resolutions} we have
\begin{align*}
&C^p(\lam)=\bigoplus_{w\in \affW^p(\lam)}
\lim\limits_{\longrightarrow \atop i}M^{-\gamma_i}(w\circ \lam)=
\bigoplus_{w\in \affW^p(\lam)}W(w\circ \lam)\quad
\text{for }p\in \Z,\\
&H^{p}(C^{\bullet}(\lam))
=\lim\limits_{\longrightarrow\atop i}H^p(B^{\bullet}_{-\gamma_i}(\lam))
=\begin{cases}
  L(\lam)&\text{if }p=0,\\ 0&\text{otherwise},
 \end{cases}
\end{align*}
and
the differential
$d_p: C^p(\lam)\ra C^{p+1}(\lam)$ has the 
form
\begin{align*}
d_p=\sum\limits_{w\in \affW^p(\lam),
\ w'\in \affW^{p+1}(\lam)\atop w\rhd_{\lam,\semiinf} w'
}d_{w',w},
\end{align*} 
where $d_{w',w}: W(w\circ \lam)\ra W(w'\circ \lam)$
is induced by
the homomorphisms
$d^{-\gamma_i}_{w',w}:M^{-\gamma_i}(w\circ \lam)\ra M^{-\gamma_i}(w'\circ
 \lam) $
with $i=1,2,\dots, $.
To complete the proof of (ii)
it remains to show that
the  map $d_{w',w}$ is nonzero for $w\rhd_{\lam,\semiinf} w'$.

Let $w',w\in \affW(\lam)$ such that $w\rhd_{\lam,\semiinf} w'$.
We have the commutative diagram
\begin{align*}
  \begin{CD}
M^{-\gamma_i}(w'\circ \lam) @>d^{-\gamma_i}_{w,w'} >> M^{-\gamma_i}(w\circ \lam)\\
@VV \phi^{w'\circ \lam}_{-\gamma_i}
 V @VV \phi^{w\circ \lam}_{-\gamma_i} V\\
W(w'\circ \lam )@>d_{w,w'} >>W(w\circ \lam)
  \end{CD}&
\end{align*}
for all $i$.
By applying the functor $G_{-\gamma_i}$
we obtain the commutative diagram 
\begin{align*}
  \begin{CD}
M^{}(t_{\gamma_i}w'\circ \lam) @>G_{-\gamma_i}(d^{-\gamma_i}_{w,w'}) >> 
M(t_{\gamma_i}w\circ \lam)\\
@VV G_{-\gamma_i}(\phi_{-\gamma_i}^{w'\circ \lam})
 V @VV  G_{-\gamma_i}(\phi_{-\gamma_i}^{w'\circ\lam}) V\\
W(t_{\gamma_i}w'\circ \lam )@>G_{-\gamma_i}(d_{w,w'}) >>W(t_{\gamma_i}w\circ \lam).
  \end{CD}&
\end{align*}
By Corollary \ref{Co:hom-space-of-Wakimoto}
$d_{w,w'}\ne 0$ if and only if
$G_{-\gamma_i}(d_{w,w'})\ne 0$.
Therefore
it is sufficient to show that
$G_{-\gamma_i}(\phi_{-\gamma_i}^{w'\circ\lam})
\circ
 G_{-\gamma_i}(d^{-\gamma_i}_{w,w'})
: M(t_{\gamma_i}w'\circ \lam)\ra 
W(t_{\gamma_i}w\circ \lam)$ is non-zero
for a sufficiently large $i$.

Write $w'=s_{\alpha}w$ with 
$\alpha\in \Delta^{re}$,
$\bar \alpha\in \sDelta_-$.
(This is possible because $s_{\alpha}=s_{-\alpha}$.)
Then, for a sufficiently large $i$,
$\beta:=t_{\gamma_i}(\alpha)\in \affrp
$
and
$t_{\gamma_i}s_\alpha w=s_{\beta}t_{\gamma_i}w\ra t_{\gamma_i}w$.
The determinant formula  \cite[Proposition 2 (2)]{Fre92}
shows that
the image of the highest weight vector
of $M(t_{\gamma_i}w'\circ \lam)=M(s_{\beta}t_{\gamma_i}w\circ \lam)$
in $M(t_{\gamma_i}w\circ \lam)$
is not in the kernel
of the map
$G_{\gamma_i}(\phi_{\gamma_i}^{w',\lam});
M(t_{\gamma_i}w\circ \lam)\ra W(t_{\gamma_i}w\circ \lam)$.
Therefore 
$G_{\gamma_i}(\phi_{\gamma_i}^{w',\lam})
\circ
 G_{\gamma_i}(d^{\gamma_i}_{w,w'})$
is non-zero,
and hence so is $d_{w,w'}$.

Finally we shall prove (i).
Note that 
\begin{align*}
 \Hom_{\g}(W(w'\circ \lam), W(w\circ \lam))
=\lim_{\longleftarrow\atop i}
\Hom_{\g} (M^{-\gamma_i}(w'\circ \lam), W(w\circ \lam))
\end{align*}
and that
$\Hom_{\g}(M^{-\gamma_i}(w'\circ \lam), W(w\circ \lam))$
is at most
 one-dimensional
by the Jantzen sum formula
since $w'\rhd_{\lam}w$.
It follows from 
 (the proof of) (ii)
that
 $\Hom_{\g} (W(w'\circ \lam), W(w\circ \lam))$
is spanned by $d_{w,w'}$.
This completes the proof.
\end{proof}

\begin{Rem}
By Theorem \ref{Th:two-sided BGG} (i)
the resolution in Theorem \ref{Th:two-sided BGG} (ii)
may be described in terms of screening operators as in \cite{BerFel90}
provided that
the existence of corresponding cycles
is established, see e.g.\ \cite{TsuKan86}.
\end{Rem}

The following assertion is an immediate consequence of Theorem
\ref{Th:two-sided BGG}
which generalizes 
\cite[Theorem 4.1]{FeuFre90}.
\begin{Th}
 Let $k$ be an admissible number,
 $\lam\in Pr_k^+$,
$p\in \Z$.
We have
\begin{align*}
& H^{\semiinf+p}(\a, L(\lam))=\bigoplus_{w\in \affW^p(\lam)}
\C_{w\circ \lam}\quad\text{as }\affh\text{-modules,}\\
& H^{\semiinf+p}(L\finn, L(\lam))=\bigoplus_{w\in \affW^p(\lam)
}\pi_{w\circ \lam+h\che \Lam_0}\quad\text{as }\mc{H}\text{-modules.}
\end{align*}
\end{Th}

\subsection{A description of vacuum admissible representation}
Let $V^{k}(\fing)$ be the 
universal affine vertex algebra associated with $\fing$ at level $k$:
\begin{align*}
V^{k}(\fing)=U(\affg)\*_{U(\fing[t])\+ \C K)
}\C_{k},
\end{align*}
where $\C_{k}$ is the one-dimensional representations of 
$\fing[t]\+ \C K$
on which $\fing[t]$ acts trivially and $K$ acts as the
multiplication by $k$.
By \cite[Proposition 5.2]{Fre05} we have an
injective homomorphism
of vertex algebras
\begin{align*}
 V^k(\fing)\hookrightarrow W(k\Lam_0)
\end{align*}
for all $k\in \C$.
Hence
$\Vg{k}$ may be regarded as a vertex subalgebra of $W(k\Lam_0)$.

Note that
$L(k\Lam_0)$ is the unique simple quotient  of $\Vg{k}$.
\begin{Pro}\label{Pro:vacuume admissible}
 Let $k$ be an admissible number,
$\Psi: W(\dot{s}_0\circ k\Lam_0)
\ra W(k\Lam_0)$
 a non-zero $\affg$-homomorphism,
which exists unique up to a nonzero constant multiplication by Theorem
 \ref{Th:two-sided BGG} (i).
Then
the image of the highest weight vector
of $W(\dot{s}_0\circ k\Lam_0)$ generates the 
maximal submodule of $\Vg{k}\subset W(k\Lam_0)$.
\end{Pro}
\begin{proof}
By
\cite{KacWak88}
the maximal submodule of $\Vg{k}$ is generated by
a  singular vector $v$
of weight $\dot{s}_0\circ k\Lam_0$.
Consider the two-sided resolution
$C^{\bullet}(k\Lam_0)$ of $L(k\Lam_0)$
in Theorem \ref{Th:two-sided BGG} (ii).
Because 
it is a resolution of $L(k\Lam_0)$ and
$\Vg{k}\subset W(k\lam_0)$,
the vector $v$ must be in the image
of $d_{1,w}: W(w\circ k\Lam_0)\ra W(k\lam_0)$
for some $w\in \affW^{-1}(k\Lam_0)$.
Since the weight  $w\circ k\Lam_0$ is strictly smaller than
$\dot{s}_0\circ k\Lam_0$ for $w\in \affW^{-1}(k\Lam_0)\backslash \{\dot{s}_0\}$,
the only possibility is that
$v$ is the image of the highest weight vector
of $W(\dot{s}_0\circ k\Lam_0)$.
\end{proof}
\subsection{Two-sided BGG resolutions of more general admissible
  representations}
 Let $\lam\in Pr_{k,y}$
with $y=\bar yt_{\eta}$,
$\bar y\in \finW$,
$\eta\in \sQche$.
Then there exists
$\lam_1\in Pr_k^+$ such that
$\lam=y\circ \lam_1$.
Since $y(\Delta(\lam_1)_+)\subset \affrp$,
$T_y: \BGG^{\g}_{[\lam_1]}\ra \BGG^{\g}_{[\lam]}$ is exact,
\begin{align*}
& T_y L(\lam_1)\cong L(\lam),\\
& T_y W(w\circ \lam_1)\cong
T_y \lim_{\longra\atop i}M^{-\gamma_i}(w\circ \lam_1)
\cong \lim_{\longra  \atop i}
T_y M^{-\gamma_i}(w\circ \lam_1)\\
&\cong \lim_{\longra \atop i} M^{-y(\gamma_i)}(ywy\inv\circ \lam)
\cong W^{\bar y}(ywy\inv \circ \lam)
\end{align*}
for $w\in \affW(\lam_1)=y\inv \affW(\lam)y$
by Proposition \ref{Pro:twisted-Wakimoto-as-limit},
Lemmas \ref{Lem:generic-twisting} and \ref{Lem:T_i M^w},
where
$(\gamma_1,\gamma_2,\dots,)$ is a sequence
as in proof of Theorem \ref{Th:two-sided BGG}. 
Therefore the following assertion follows immediately 
from Theorem \ref{Th:twisted-BGG}.

\begin{Th}\label{Th:general-twisted-BGG}.
Let $k$ be an admissible number,
 $\lam\in Pr_{k,y}$
with  
$y=\bar yt_{\eta}$,
$\bar y\in \finW$,
$\eta\in \sPche$.
Then 
there exists 
 a complex 
\begin{align*}
 C^{\bullet}(\lam): \cdots \overset{d_{-3}}{\rightarrow} C^{-2}(\lam)
\overset{d_{-2}}{\rightarrow} C^{-1}(\lam)
\overset{d_{-1}}{\rightarrow} C^0(\lam)
\overset{d_0}{\rightarrow} C^1(\lam)
\overset{d_1}{\rightarrow} C^2(\lam)\overset{d_2}{\rightarrow} \cdots
\end{align*}
in the category $\BGG$
of the form
$C^i=\bigoplus\limits_{w\in \affW^i(\lam)
}
W^{\bar y}(w\circ \lam)$,
$d_i=\sum\limits_{
w\in \affW^i(\lam),\ w'\in \affW^{i+1}(\lam)
\atop w\rhd_{\lam,\semiinf} w'}
d_{w',w}
 $.
such that
\begin{align*}
 H^i(C^{\bullet}(\lam))\cong \begin{cases}
			      L(\lam)&\text{for } i=0,\\ 0&\text{for }i\ne 0.
			     \end{cases}
\end{align*}
\end{Th}

 \begin{Rem}
If $\lam\in Pr_{k,y}$ and $\bar y=1$ (that is, $y\in \sP\che$), then
$W^{\bar y}(w\circ \lam)=W(w\circ \lam)$.
Hence the above is the resolution of $L(\lam)$ in terms of (non-twisted)
  Wakimoto modules as conjectured in \cite{FKW92}.
 \end{Rem}

\section{Semi-infinite restriction and induction}
\label{section:Semi-infnite <restriction}
\subsection{Feigin-Frenkel   parabolic  induction}
Let $\finp$
 be a parabolic subalgebra 
of $\fing$
containing
$\finb_-$,
and 
let
$\finp=\finl\+\finm_-$
be the direct sum decomposition of $\finp$
with the Levi subalgebra $\finl$  containing $\finh$
and the nilpotent radical 
 $\finm_-$.
Denote by  $\finm\subset \finn$ the opposite algebra of $\finm_-$,
so that
$\fing=\finp\+\finm$.
Let
\begin{align*}
 \finl=\finl_0\+\bigoplus_{i=1}^s \finl_i
\end{align*}
be the decomposition of $\finl$ into 
direct sum of 
simple Lie subalgebras $\finl_i$,
$i=1,\dots, s$, and its center $\finl_0$ of $\finl$.
Let $\finh_i=\finl\cap \finh$,
the Cartan subalgebra of $\finl_i$,
and 
denote by
$\sDelta_i\subset \sDelta$ the subroot system of $\fing$ corresponding
 to $\finl_i$,
$\sPi_i=\sPi\cap \sroots_i$.
Let  $h_i\che$ be the dual Coxeter number of $\finl_i$ (with a convention
$h_0\che=0$),
  $\theta_i$ the highest  root of $\sDelta_i$,
$\theta_{i,s}$ 
the highest short roof of $\sDelta_i$.

Let  $\affl_i=\finl_i[t,t\inv]\+ \C K\subset \affg$
for 
$i=0,1,\dots,s$.
Set
\begin{align*}
 K_i=\frac{2}{(\theta_i|\theta_i)}K,
\end{align*}
and we consider $K_i$ as an element of $\affl_i$.
Thus,
\begin{align*}
\affl_i=\finl_i[t,t\inv]\+ \C K_i,
\end{align*}
and
$\affh_i:=\finh_i\+ \C K_i$ is a Cartan subalgebra
of $\finl_i$.

Define
\begin{align*}
 \affl=\bigoplus_{i=0}^s \affl_i.
\quad \afft=\bigoplus_{i=0}^s\affh_i.
\end{align*}
The grading of $\affl_i$ induces the grading of
$\affl$.

For $k\in \C$
define 
$k_0,\dots, k_s\in \C$ by
\begin{align}
k_0=k+h\che,\quad
k_i+h_i\che=
\frac{2}{(\theta_i|\theta_i)}( k+h\che)\quad\text{for }i=1,\dots,s.
\label{eq:level}
\end{align}
 \begin{Lem}\label{Lem:ki-is-admissble}
Let $k$ be an admissible  number for $\affg$.
Then $k_i$,
$i=1,\dots,s$,
is an admissible number for the Kac-Moody algebra $\affl_i$.
 \end{Lem}

Let
 $\BGG^{\affl}_{(k_0,\dots,k_s)}$
be the 
 full subcategory  of 
$\BGG^{\affl}$ consisting of objects
on which $K_i$ acts as the multiplication by $k_i$,
$i=0,1,\dots,s$.
Feigin and Frenkel \cite[5.2]{FeuFre90},
\cite[\S 6]{Fre05}
constructed a functor
\begin{align*}
\on{F-ind}^{\affg}_{\affl}: \BGG_{(k_0,k_1,\dots,k_s)}^{\affl}\ra \BGG_k^{\affg},
\quad M\ra \on{F-ind}^{\affg}_{\affl}(M),
\end{align*}
which  enjoys the property
\begin{align}
 \on{F-ind}^{\affg}_{\affl}
(M)\cong \SS{L\finm}\*_{\C}M
\label{eq:frenkel-induction}
\end{align}
as modules over 
\begin{align*}
L\finm=\finm[t,t\inv]\subset \affg,
\end{align*}
where $L\finm$ only on the first factor $\SS{L\finm}$.
In particular
$ \on{F-ind}^{\affg}_{\affl}$ is an exact functor.

Denote by 
$W_{\affl_i}(\lam^{(i)})$ 
the Wakimoto module of the affine Kac-Moody algebra
$\affl_i$
with highest weight $\lam^{(i)}\in \dual{\affh_i}$
and by
$L_{\affl}(\lam^{(i)})$ the irreducible highest weight representation
of $\affl_i$
with highest weight $\lam^{(i)}$
(with a convention  that 
 $W_{\affl_0}(\lam^{(0)})$  is the irreducible representation 
of the Heisenberg algebra $\affl_0$ with highest weight
$\lam^{(0)}$).
For $\lam\in \dual{\afft}$
let 
$W_{\affl}(\lam)$ and
$L_{\affl}(\lam)$ 
be the Wakimoto module and the irreducible highest weight 
representation
of $\affl$
with highest weight $\lam$:
\begin{align*}
 W_{\affl}(\lam)=\bigotimes_{i=0}^s W_{\affl_i}(\lam|_{\affh_i}),\quad
 L_{\affl}(\lam)=\bigotimes_{i=0}^s L_{\affl_i}(\lam|_{\affh_i}).
\end{align*}

For $\lam\in \dual{\affh}$,
define $\lam_{\affl}\in \dual{\afft}$ by
\begin{align*}
\lam_{\affl}|_{\finh_i}=\lam|_{\finh_i}
\text{ and }
(\lam_{\affl}+\rho_i)(K_i)=\frac{2}{(\theta_i|\theta_i)}(\lam+\rho)(K)
\end{align*}for $i=0,1,\dots,s$.
\begin{Pro}[\cite{FeuFre90}]\label{Pro:Wakimoto-as-parabolic-induction}
For $\lam\in \dual{\affh}$
we have
$ 
\on{F-ind}^{\affg}_{\affp}W_{\affl}(\lam_{\affl})
\cong W(\lam)$.
\end{Pro}
 \begin{proof}
By 
using the 
Hochschild-Serre
 spectral sequence
for $L\finm\subset \a$
we see from \eqref{eq:frenkel-induction}
that
\begin{align*}
H^{\semiinf+i}(\a, \on{F-ind}^{\affg}_{\affl}W_{\affl}(\lam_{\affl}))
\cong \begin{cases}
       \C_{\lam}&\text{for }i=0,\\0&\text{otherwise.}
      \end{cases}
\end{align*}Hence the assertion follows from Theorem \ref{Th:uniqueness-of-Wakimoto}.
 \end{proof}
\subsection{Semi-infinite restriction functors}
Let $M\in \BGG_k^{\affg}$.
Then
$H^{\semiinf+p}(L\finm, M)$,
$p\in \Z$,
is naturally an $\affl$-module on which $K_i$ acts as the multiplication
by $k_i$,
see e.g.\  
 \cite[Proposition 2.3]{HosTsu91}.
Hence
\begin{align*}
\on{S-res}^{\affg}_{\affl}
:=H^{\semiinf+0}(L\finm, ?)
\end{align*}
defines a functor
$\BGG^{\affg}_k\ra \BGG^{\affl}_{(k_0,k_1,\dots,k_s)}$.
We refer to $\on{S-res}^{\affg}_{\affl}$ as the 
{\em semi-infinite
restriction functor}.

The following assertion 
follows from  Proposition \ref{Pro:Wakimoto-as-parabolic-induction}.
\begin{Pro}\label{Pro:restriction-of-Wakimoto}
For $\lam\in \dual{\affh}$
we have
$H^{\semiinf+i}(L\finm,W(\lam))=0$ for $i\ne 0$ 
and 
\begin{align*}
\on{S-res}^\affg_{\affl}W(\lam)\cong W_{\affl}(\lam_{\affl}).
\end{align*} 
\end{Pro}

\subsection{Decomposition of integral Weyl groups}
\label{subsection:parabolic}
Let
$k$ be an admissible number
with denominator $q$,
$\lam\in Pr_k^+$.
Let $\finW_{S_i}$ be the 
parabolic subgroup of $\finW$
corresponding to
$\finl_i$,
$\finW_{S}=\finW_{S_1}\times \finW_{S_2}\times \dots \times
\finW_{S_s}$.
Define $\dot{\alpha}_0^{(i)}
\in \Delta(\lam)$,
$i=1,\dots,s$,
by
\begin{align*}
 &\dot{\alpha}_0^{(i)}
=-\theta_i+q\delta\quad \text{if }(r\che,q)=1,\\
\text{ and }&(\dot{\alpha}_0^{(i)})\che=-\theta_{i,s}\che+q\delta\quad\text{ if }
(r\che,q)=r\che.
\end{align*}
Set $\dot{s}_0^{(i)}=s_{\dot{\alpha}_0^{(i)}}$.

Let
$\affW(\lam)_{S_i}$ be the subgroup
of $\affW(\lam)$ generated by
$\finW_{S_i}$ and $\dot{s}_0^{(i)}$.
Then 
\begin{align*}
\affW(\lam)_{S}=\affW(\lam)_{S_1}\times \affW(\lam)_{S_2}\times \dots \times
\affW(\lam)_{S_s}
\end{align*} is the subgroup corresponding to
$\finW_{S}$ described 
in \S \ref{subsection:Semi-infinite analogue of
parabolic subgroups}.
Let $\affW(\lam)^S\subset \affW(\lam)$
be as in  Theorem \ref{Th:semi-infinte-minimal-length-representative}
so that
\begin{align}
 \affW(\lam)=\affW(\lam)_S\times \affW(\lam)^S,
\quad
\ell^{\semiinf}_{\lam}(uv)=\ell^{\semiinf}_{\lam}(u)+\ell^{\semiinf}_{\lam}(v)
\text{ for }u\in \affW(\lam)_S,\ v\in \affW(\lam)^S.
\label{eq:miminal-length-integral}
\end{align}

Let $w,w'\in \affW(\lam)_{S_i}\subset \affW(\lam)$
such that
$w\rhd_{\lam,\semiinf}w'$.
Then
$w\circ_{\affl_i} \lam_{\affl}^{(i)}=
(w\circ \lam)_{\affl}^{(i)}$,
where
$\circ_{\affl_i}$ is the dot action of $\affW(\lam)_{S_i}$
on $\dual{\affh}_i$
and $\lam_{\affl_i}^{(i)}=\lam_{\affl}|_{\affh_i}$.

 \begin{Pro}
\label{Pro:restriction-of-hom-between-Wakimoto}
Let $\lam\in Pr_k^+$,
$w,w'\in \affW(\lam)_{S_i}$
with $i\in \{1,2,\dots, s\}$
such that $w\rhd_{\lam,\semiinf} w'$.
Then the correspondence
$\Phi\mapsto \on{F-ind}^{\affg}_{\affl}(\Phi)$
defines a linear isomorphism
\begin{align*}
 \Hom_{\affl}(W_{\affl}((w\circ \lam)_{\affl}
 ),
W_{\affl}((w'\circ \lam)_{\affl})
\isomap  \Hom_{\affg} (W(w \circ \lam),W(w'\circ \lam)).
\end{align*}
The inverse map is given by
$\Psi\ra \on{S-res}^{\affg}_{\affl}(\Psi)$.
 \end{Pro}
 \begin{proof}
By Proposition \ref{Pro:endmorphisms-of-Wakimoto}
and
 Theorem \ref{Th:two-sided BGG} (i)
both 
$\Hom_{\affl}(W_{\affl}((w\circ \lam)_{\affl}
 ),
W_{\affl}((w'\circ \lam)_{\affl})$
and
$\Hom_{\affg} (W(w \circ \lam),W(w'\circ \lam))$
are one-dimensional.
The assertion follows since  the correspondence $\Phi\mapsto
  \on{F-ind}^{\affg}_{\affl}(\tilde{\Phi})$
is clearly injective
and 
$\on{S-res}^{\affg}_{\affl}(\on{F-ind}^{\affg}_{\affl}(\Phi))
=\Phi$.
 \end{proof}

\subsection{Semi-infinite restriction of admissible affine vertex algebras}
Since it is defined by the semi-infinite cohomology
the space
$\on{S-res}^{\affg}_{\affl}(\Vg{k})$
inherits a vertex algebra
structure from
$\Vg{k}$, and 
we have a natural vertex algebra homomorphism
\begin{align*}
 \bigotimes_{i=0}^s V^{k_i}(\finl_i)
\ra 
\on{S-res}^{\affg}_{\affl}(\Vg{k}),
\end{align*}
where $V^{k_i}(\finl_i)$ denote the universal affine
vertex algebra associated with
$\finl_i$ at level $k_i$.
By composing with the map
$\on{S-res}^{\affg}_{\affl}(\Vg{k})\ra 
\on{S-res}^{\affg}_{\affl}(\Vs{k})$ 
induced by the surjection
$\Vg{k}\twoheadrightarrow \Vs{k}$
this gives rise to
 a vertex algebra homomorphism
\begin{align}
 \bigotimes_{i=0}^s V^{k_i}(\affl_i)
\ra
\on{S-res}^{\affg}_{\affl}(\Vs{k}).
\label{eq:va homo}
\end{align}
On the other hand
there is a natural surjective homomorphism
\begin{align*}
\bigotimes_{i=0}^s V^{k_i}(\affl_i)
\twoheadrightarrow
\bigotimes_{i=0}^s L_{\affl_i}(k_i \Lam_{0})
\end{align*}
of vertex algebras,
where
$L_{\affl_i}(k_i \Lam_{0})$ is the unique simple quotient 
of $V^{k_i}(\affl_i)$.
\begin{Th}\label{Th:can-reduce-to-sl_2}
Let $k$ be an admissible number.
 The vertex algebra homomorphism
(\ref{eq:va homo})
factors through the
vertex algebra homomorphism
\begin{align*}
\bigotimes_{i=0}^s L_{\affl_i}(k_i \Lam_{0})
\hookrightarrow 
\on{S-res}^{\affg}_{\affl}( \Vs{k}).
\end{align*}\end{Th}
\begin{proof}
Put $\lam=k\Lam_0$
and 
let
$C^{\bullet}(\lam)$ be the two-sided BGG resolution
of $\Vs{k}$
in Theorem \ref{Th:two-sided BGG}.
By the vanishing assertion of Proposition \ref{Pro:restriction-of-Wakimoto}
 the semi-infinite cohomology 
$H^{\semiinf+\bullet}(L\finm, L(\lam))$ is 
isomorphic
to the cohomology of the complex
$\on{S-res}^{\affg}_{\affl}(C^{\bullet}(\lam))$
obtained from $C^{\bullet}(\lam)$
applying the functor
$\on{S-res}^{\affg}_{\affl}$.
Thus $\on{S-res}^{\affg}_{\affl}( \Vs{k})$ is isomorphic to the zero-th
 cohomology of the complex $\on{S-res}^{\affg}_{\affl}(C^{\bullet}(\lam))$.

Consider the map
$C^{-1}(\lam)\supset W(\dot{s}_0^{(i)}
\circ \lam)
\overset{d_{1,\dot{s}_0^{(i)}}}{\ra} W(\lam)\subset C^0(\lam)$
for $i=1,\dots, s$.
By applying the functor
$\on{S-res}^{\affg}_{\affl}$
this
induces a non-zero homomorphism
\begin{align*}
W_{\affl}(\dot{s}_0^{(i)}\circ_{\affl_i}\lam_{\affl})
\ra W_{\affl}(\lam_{\affl})
\end{align*}by Proposition \ref{Pro:restriction-of-hom-between-Wakimoto},
and the
image of the highest weight vector of
$W_{\affl}(\dot{s}_0^{(i)}\circ_{\affl_i}\lam_{\affl})$
generates the
maximal  $\affl_i$-submodule 
of $V^{k_i}(\affl_i)\subset W_{\affl}(\lam_{\affl})$
by Proposition \ref{Pro:vacuume admissible}.
It follows that
 the maximal $\affl$-submodule of
$ \bigotimes_{i=0}^s V^{k_i}(\affl_i)\subset  
W_{\affl}(\lam)$ is in the image of 
$\on{S-res}^{\affg}_{\affl}(d_{-1}): 
\on{S-res}^{\affg}_{\affl}(C^{-1}(\lam))
\ra \on{S-res}^{\affg}_{\affl}(C^{0}(\lam))$.
This completes the proof.
\end{proof}
\subsection{The case of minimal parabolic subalgebras}
Consider the case
that $\finp$ is generated by $\finb_-$ and $e_i$
with $i\in \sI$.
Then $\finl=\finl_0\+ \finl_1$,
 $\finl_1=\mf{sl}_2^{(i)}$ and 
$\affl_1=\wh{\mf{sl}}_2^{(i)}$.
 \begin{Th}[$\finp$ minimal]
\label{Th:direct-sum-of-sl2}
Let $k$ be an admissible number and let $M$ be a module over
the vertex algebra $\Vs{k}$.
Then, for each $p\in \Z$,
$H^{\semiinf+p}(L\finm,M)$
is a direct sum of admissible representations 
of level $k_1$ (see \eqref{eq:level}) as 
  $\wh{\mf{sl}}_2^{(i)}$-modules.
 \end{Th}
 \begin{proof}
By Theorem \ref{Th:can-reduce-to-sl_2},
$L_{\affl_1}(k_1\Lam_0)$ is a vertex subalgebra
of $\on{S-res}^{\affg}_{\affl}(\Vs{k})
=H^{\semiinf+0}(L\finm, \Vs{k})$.
If
$M$ is a 
module over $L(k\Lam_0)$ then 
$H^{\semiinf+p}(L\finm, M)$
 is naturally a module 
over $\on{S-res}^{\affg}_{\affl}(\Vs{k})$,
and therefore,  it is a module over $L_{\affl_1}(k_1\Lam_0)$.
The assertion follows since it is known
by \cite{AdaMil95}
that
any module over $L_{\affl_1}(k_1\Lam_0)$
in the category $\BGG^{\affl_1}$
 must be a direct sum of admissible representations
of $\affl_1\cong \wh{\mf{sl}}_2$.
 \end{proof}

The following assertion generalizes
\cite[Theorem 3.8]{HosTsu91}
in the case that $\finp$ is minimal.
 \begin{Th}[$\finp$ minimal]
\label{Th:generalized-bore-weil}
Let $k$ be an admissible number,
$\lam\in Pr_k^+$.
Then
\begin{align*}
 H^{\semiinf+p}(L\finm, L(\lam))\cong 
\bigoplus_{w\in\affW(\lam)^S\atop
\ell^{\semiinf}(w)=p}L_{\affl}((w\circ \lam)_{\affl})
\end{align*}
as $\affl$-modules.
 \end{Th}
 \begin{proof}
It is known by  \cite{MalFre99}
(see also \cite{FreMal97})
that 
$L(\lam)$ with $\lam\in Pr_k^+$
is a module over $\Vs{k}$.
Therefore
$H^{\semiinf+\bullet}(L\finm, L(\lam))$
is a direct sum of
irreducible 
admissible representations
as $\wh{\mf{sl}}_2^{(i)}$-modules
by
Theorem \ref{Th:direct-sum-of-sl2}.
Hence
 it is sufficient to determine the subspace
$H^{\semiinf+\bullet}(L\finm, L(\lam))^{\affl_+}$
of the singular vectors of
$H^{\semiinf+\bullet}(L\finm, L(\lam))$.
Clearly,  any weight of 
$H^{\semiinf+\bullet}(L\finm, L(\lam))^{\affl_+}$
must be admissible for $\affl_1=\widehat{\mf{sl}}_2^{(i)}$.

As is remarked in the proof of Proposition \ref{Th:can-reduce-to-sl_2},
$ H^{\semiinf+\bullet}(L\finm, L(\lam))$
is the cohomology
of the
complex $\on{S-res}^{\affg}_{\affl}(C^{\bullet}(\lam))$
and we have
$\on{S-res}^{\affg}_{\affl}(C^{p}(\lam))
=\bigoplus_{w\in \affW^p(\lam)}W_{\affl}((w\circ\lam)_{\affl})$
by Proposition \ref{Pro:restriction-of-Wakimoto}.
Now Theorem  \ref{Th:semi-infinte-minimal-length-representative}
and Lemma \ref{Lem:ki-is-admissble}
imply that
\begin{align*}
&\{(w\circ \lam)_{\affl}; w\in \affW(\lam),
\ (w\circ \lam)_{\affl}\text{ is an admissible weight
for }\wh{\mf{sl}}_2^{(i)}\}\\
=&\{(w\circ \lam)_{\affl}; w\in \affW(\lam),
\ (w\circ \lam)_{\affl}\text{ is a dominant  weight
for }\wh{\mf{sl}}_2^{(i)}
\}\\
=&\{(w\circ \lam)_{\affl}; w\in \affW(\lam)^S\}.
\end{align*}
It follows that
if a weight $\mu$ of $W_{\affl}((w\circ \lam)_{\affl})$ is
admissible for $\wh{\mf{sl}}_2^{(i)}$
then $w\in \affW(\lam)^S$ and $\mu= (w\circ \lam)_{\affl}$.
Therefore
 the image $[|{(w\circ \lam)_{\affl}}\ket 
]$ of 
the highest weight vector  $|{(w\circ \lam)_{\affl}}\ket$ of
$\W_{\affl}((w\circ \lam)_{\affl})$
is nonzero in $H^{\semiinf+\bullet}(L\finm, L(\lam))$
and
$\{[|{(w\circ \lam)_{\affl}}\ket ]; w\in \affW(\lam)^S
\}$
forms a basis of 
$H^{\semiinf+\bullet}(L\finm, L(\lam))^{\affl_+}$.
By Theorem  \ref{Th:semi-infinte-minimal-length-representative},
this completes the proof.
 \end{proof}
 \begin{Rem}
In the subsequent paper \cite{A12-2}
we prove that
for an admissible number $k$
any $\Vs{k}$-module in the category
$\BGG^{\g}$ must be a direct sum of admissible representations.
Hence
it follows from the proof that
the assertion of 
Theorem \ref{Th:generalized-bore-weil} is valid
for any parabolic subalgebra of $\fing$.
 \end{Rem}
\bibliographystyle{alpha}
\bibliography{math}

\end{document}